\documentclass[10pt, leqno, a4paper]{amsart}

\usepackage{esint} 
\usepackage{amsmath, amsthm}
\usepackage{mathrsfs}    
\usepackage{amssymb}
\usepackage{amsmath, amsthm}
\usepackage{mathrsfs}    
\usepackage{amssymb}
\usepackage{dsfont}

\usepackage{graphicx,color}

\usepackage[colorlinks]{hyperref}
\hypersetup{linkcolor=blue,citecolor=blue,filecolor=black,urlcolor=blue}

\usepackage{verbatim} 

\usepackage[colorlinks]{hyperref}
\hypersetup{linkcolor=blue,citecolor=blue,filecolor=black,urlcolor=blue}

\usepackage[showonlyrefs]{mathtools}
\mathtoolsset{showonlyrefs=true}

\usepackage[sort,nocompress]{cite}

\usepackage[sort,nocompress]{cite}

\newtheorem{theorem}{Theorem}[section]
\newtheorem{proposition}[theorem]{Proposition}
\newtheorem{corollary}[theorem]{Corollary}
\newtheorem{lemma}[theorem]{Lemma}

\theoremstyle{definition}

\newtheorem{remark}[theorem]{Remark}

\usepackage[
  hmarginratio={1:1},     
  vmarginratio={1:1},     
  textwidth=15cm,        
  textheight=20cm,  
  heightrounded,          
]{geometry}

\numberwithin{equation}{section}

\def\eps{\varepsilon}

\def\R{\mathbb{R}}

\def\supp{\mathrm{supp\,}}
\def\dist{\mathrm{dist}}

\def\hbar{\bar{h}}

\def\le{\leqslant}
\def\ge{\geqslant}

\newcommand{\bu}{\mathbf{u}}
\newcommand{\id}{{\rm Id}}

\def\pa{\partial}

\newcommand{\rst}[1]{\ensuremath{{\mathbin |}%
\raise-.5ex\hbox{$#1$}}}
\newcommand{\mf}[1]{\mathbf{#1}}

\let\div\relax
\DeclareMathOperator{\div}{div}

\definecolor{verde}{rgb}{0,0.35,0.1} 
\definecolor{rosso}{rgb}{0.7,0,0}
\definecolor{blue}{rgb}{0,0,1}
\definecolor{viola}{rgb}{0.6,0,0.4}
\definecolor{grigio}{rgb}{0.5,0.5,0.5}

\author[N. Soave]{Nicola Soave}\thanks{}
\address{Nicola Soave \newline \indent
Dipartimento di Matematica,  Politecnico di Milano,  \newline \indent
Via Edoardo Bonardi 9, 20133 Milano, Italy}
\email{nicola.soave@gmail.com; nicola.soave@polimi.it}

\author[H. Tavares]{Hugo Tavares}\thanks{}
\address{Hugo Tavares \newline \indent CAMGSD (Center for Mathematical Analysis, Geometry and Dynamical Systems) \newline\indent
Departamento de Matem\'atica, Instituto Superior T\'ecnico, Universidade de Lisboa\newline \indent Av. Rovisco Pais, 1049-001 Lisboa, Portugal\newline
\indent {\small{and}} \newline \indent Departamento de Matem\'atica, Faculdade de Ci\^encias da Universidade de Lisboa \newline \indent Edif\'icio C6, Piso 1, Campo Grande 1749-016  Lisboa, Portugal}
\email{htavares@math.ist.utl.pt; hrtavares@ciencias.ulisboa.pt}

\author[S. Terracini]{Susanna Terracini}\thanks{}
\address{Susanna Terracini \newline \indent
 Dipartimento di Matematica ``Giuseppe Peano'', Universit\`a di Torino, \newline \indent
Via Carlo Alberto, 10,
10123 Torino, Italy}
\email{susanna.terracini@unito.it}

\author[A. Zilio]{Alessandro Zilio}\thanks{}
\address{Alessandro Zilio \newline \indent  \'{E}cole des Hautes \'{E}tudes en Sciences Sociales, PSL Research University Paris, Centre d'analyse et \indent de math\'{e}matique sociales (CAMS), CNRS, \newline \indent 190-198 Avenue de France, 75244, Paris CEDEX 13, France}
\email{azilio@ehess.fr, alessandro.zilio@polimi.it}

\title{Variational problems with long-range interaction}
\date{\today}

\begin{document}
\maketitle

\begin{abstract}
We consider a class of variational problems for densities that repel each other at distance. Typical examples are given by the Dirichlet functional and the Rayleigh functional
\[
	D(\mf{u}) =  \sum_{i=1}^k \int_{\Omega} |\nabla u_i|^2 \quad \text{or} \quad R(\mf{u}) =  \sum_{i=1}^k \frac{\int_{\Omega} |\nabla u_i|^2}{\int_{\Omega} u_i^2}
\]
minimized in the class of $H^1(\Omega,\R^k)$ functions attaining some boundary conditions on $\partial \Omega$, and subjected to the constraint  
\[
	\dist (\{u_i > 0\}, \{u_j > 0\}) \ge 1 \qquad \forall i \neq j.
\]
For these problems, we investigate the optimal regularity of the solutions, prove a free-boundary condition, and derive some preliminary results characterizing the free boundary $\partial \{\sum_{i=1}^k u_i > 0\}$.
\end{abstract}

\section{Introduction}

The object of this paper is the study of a class of minimal configurations for variational problems involving arbitrarily many densities related by long-range repulsive interactions. The mathematical setting we consider is described by the following two archetypical situations.

\medskip 

\noindent \textbf{Problem (A)} Let $\Omega$ be a bounded domain of $\R^N$, $N \ge 2$, and let 
\[
\Omega_1=\bigcup_{x \in \Omega} B_1(x)=\{x\in \R^N:\ \dist (x,\Omega)<1\}.
\] 
Given $k \ge 2$ nonnegative nontrivial functions $f_1,\ldots, f_k\in H^1(\Omega_1)\cap C(\overline\Omega_1)$ satisfying
\footnote{Here and in the rest of the paper, the distance between two sets $A$ and $B$ is understood as
\begin{equation}\label{def dist}
\dist(A,B):= \inf\{ |x-y|: \ x \in A, \ y \in B\}.
\end{equation}
}
\[
\dist(\supp f_i,\supp f_j)\ge 1\qquad \forall i\neq j, 
\]
we consider the minimization problem 
\[
\inf_{\mf{u}\in H_\infty} J_\infty(\mf{u}),
\] where the set $H_\infty$ and the functional $J_\infty$ are defined by
\begin{equation}\label{def H infty}
H_\infty=\left\{\mf{u} = (u_1,\dots,u_k) \in H^1(\Omega_1,\R^k)\left| \begin{array}{l} \dist(\supp u_i,\supp u_j)\ge 1\quad  \forall i\neq j \\ 
u_i = f_i \ \text{a.e. in $\Omega_1 \setminus \Omega$}
\end{array}\right. \right\},
\end{equation}
and
 \begin{equation}\label{def J infty}
J_\infty(\bu)=\sum_{i=1}^k \int_\Omega |\nabla u_i|^2.
\end{equation}
The support of each component $u_i$ is taken in the weak sense: it corresponds to the complement in $\Omega_1$ of the largest open set $\omega\subseteq \R^N$ where $u_i=0$ a.e. on $\omega$ (cf. \cite[Proposition 4.17]{Brezis_book_2011}). Notice also that the existence of $f_1,\dots,f_k$ with the above properties imposes some conditions on $\Omega$ (for instance, the diameter of $\Omega$ cannot be too small), and we suppose that such conditions are satisfied.

We are interested in existence and qualitative properties of minimizers.

\medskip

\noindent \textbf{Problem (B)} Let $\Omega$ be a bounded domain of $\R^N$, $N \ge 2$, and let $k \ge 2$. We consider the \emph{set of open partitions of $\Omega$ at distance $1$}, defined as
\[
	\mathcal{P}_{k}(\Omega) = \left\{(\omega_1, \dots, \omega_k) \left| \begin{array}{l} \text{$\omega_i\subset \Omega$ is open and non-empty for every $i$}, \\
	\text{and $\dist(\omega_i, \omega_j) \ge 1 \quad \forall i \neq j$}
\end{array}\right. \right\}.
\]
Then, for a cost function $F \in \mathcal{C}^1((\R^+)^k, \R)$ satisfying
\begin{itemize}
	\item $\partial_i F(x) > 0$ for all $x \in (\R^+)^k$ and $i=1, \dots, k$, which in particular yields that $F$ is component-wise increasing;
	\item for any given $i = 1, \dots, k$,
	\[
		\lim_{x_i \to +\infty} F(\bar x_1, \dots,\bar x_{i-1},x_i, \bar x_{i+1} \dots, \bar x_k ) = +\infty
	\]
	for all $(\bar x_1, \dots,\bar x_{i-1}, \bar x_{i+1} \dots, \bar x_k ) \in (\R^+)^{k-1}$,
\end{itemize}
we consider the minimization problem
\begin{equation}\label{eqn part dist}
	\inf_{ (\omega_1, \dots, \omega_k) \in \mathcal{P}_{k}(\Omega) } F(\lambda_1(\omega_1), \dots, \lambda_1(\omega_k)),
\end{equation}
where $\lambda_1(\omega)$ is the first eigenvalue of the Laplace operator in $\omega$ with homogeneous Dirichlet boundary conditions. Problem \eqref{eqn part dist} is a particular case of an \emph{optimal partition problem} (cf. \cite{BourdinBucurOudet, BucurButtazzoHenrot}). A typical case we have in mind is the cost function $F(\lambda_1(\omega_1), \ldots, \lambda_1(\omega_k))=\sum_{i=1}^k \lambda_1(\omega_i)$.

We are interested in existence and qualitative properties of an optimal partition.

\medskip

Our main results are, for problem (A):
\begin{itemize}
\item the existence of a minimizer;
\item the optimal interior regularity of any minimizer;
\item the derivation of several properties of the positivity sets $\{u_i>0\}$;
\item the derivation of a free boundary condition involving the normal derivatives of different components of any minimizers on the regular part of the free-boundary $\pa \{u_i>0\}$.
\end{itemize}
For problem (B):
\begin{itemize}
\item the introduction of a weak formulation in terms of densities, and the existence of weak solutions;
\item the global optimal regularity of any weak solution, which leads in particular to the existence of a strong solution for the original problem;
\item the derivation of properties of the subsets $\omega_i$, and of a free boundary condition on the regular part of $\pa \omega_i$.
\end{itemize}
In a forthcoming paper, we will study more in detail the regularity of the free-boundary.
 
We stress that, both in problems (A) and (B), the interaction among different densities takes place at distance: in problem (A) the positivity sets $\{u_i>0\}$, and in problem (B) the open subsets $\omega_i$, are indeed forced to stay at a fixed minimal distance from each other. 

\medskip

When the interaction among the densities takes place point-wisely, segregation problems analogue to (A) and (B) have been studied intensively, in connection with optimal partition problems for Laplacian eigenvalues \cite{ctv2002, CoTeVe2003, CTV_fucik, CaffLin2007, rtt, ttAIHP, tt}, with the regularity theory of harmonic maps into singular manifold \cite{CaffLin_2008, CoTeVe05, tt}, and with segregation phenomena for systems of elliptic equations arising in quantum mechanics driven by strong competition \cite{CaffLin_2008, DaWaZh, NoTaTeVe, SoTaTeZi, SoZi, SoZi2, WeiWeth}. 

In contrast, the only results available so far regarding segregation problems driven by long-range competition are given in \cite{CaPaQu}, where the authors analyze the spatial segregation for systems of type
\begin{equation}\label{syst p}
\begin{cases}
-\Delta u_{i,\beta} = -\beta u_{i,\beta}\sum_{j \neq i} (\mathds{1}_{B_1} \star |u_j|^p)  & \text{in $\Omega$} \\
u_{i,\beta} = f_i \geq 0 & \text{in $\Omega_1 \setminus \Omega$},
\end{cases}
\end{equation}
with $1 \le p \le +\infty$. In the above equation, $\mathds{1}_{B_1}$ denotes the characteristic function of $B_1$, the ball\footnote{We denote by $B_r(x)$ the ball of center $x$ and radius $r$ in $\R^N$. In case $x=0$, we simply write $B_r$.} of center $0$ and radius $1$, and $\star$ stays for the convolution for $p<+\infty$, so that
\[
(\mathds{1}_{B_1} \star |u_j|^p)(x) = \int_{B_1(x)} |u_j(y)|^p\,dy \qquad \forall x \in \Omega, \ \text{with } 1 \le p < +\infty;
\]
in case $p=+\infty$, we intend that the integral is replaced by the supremum over $B_1(x)$ of $|u_j|$. In \cite{CaPaQu}, the authors prove the equi-continuity of families of viscosity solutions $\{\mf{u}_\beta: \beta > 0\}$ to \eqref{syst p}, the local uniform convergence to a limit configuration $\mf{u}$, and then study the free-boundary regularity of the positivity sets $\{u_i>0\}$ in cases $p=1$ and $p=+\infty$, mostly in dimension $N=2$. As we shall see, our problem (A) is strictly related with the asymptotic study of the solutions to \eqref{syst p} in case $p=2$ (see the forthcoming Theorem \ref{prop:existence}); nevertheless, also in such a situation our approach is very different with respect to the one in \cite{CaPaQu}, since we heavily rely on the variational nature of the problem. This  gives differenti free boundary conditions which requires different techniques, and allows us to prove new results. 

Regarding problem \eqref{syst p}, we also refer to \cite{Bozo}, where the author proves uniqueness results in the cases $p=1$ and $p=+\infty$.

\subsection{Main results} We adopt the notation previously introduced. First of all, we have the following existence results for problems (A) and (B).

\begin{theorem}[Problem (A)]\label{thm: existence}
There exists a minimizer $\mf{u}=(u_1,\ldots, u_k)$ for $\inf_{H_\infty} J_\infty$.
\end{theorem}

\begin{theorem}[Problem (B)]\label{thm: existence2}
There exists a minimizer $(\omega_1,\ldots, \omega_k)\in \mathcal{P}_k$ for \eqref{eqn part dist}.
\end{theorem}


Observe that, to each optimal partition $(\omega_1,\ldots, \omega_k)$, we can associate a vector of signed first eigenfunctions. To fix ideas, from now on we always consider nonnegative eigenfunctions. The second part of our analysis concerns the properties satisfied by any minimizer of problems (A) and (B).

\begin{theorem}\label{thm: properties minimizers}
Let $\bu=(u_1,\ldots, u_k)$ be either any minimizer of $J_\infty$ in $H_\infty$, or a vector of first eigenfunctions associated to an optimal partition $(\omega_1,\ldots, \omega_k)$ of \eqref{eqn part dist}. Then $\bu$ is a vector of nonnegative functions in $\Omega$, and denoting by $S_i$ the positivity set $\{x\in \Omega:\ u_i>0\}$, for every $i=1,\dots,k$, we have:
\begin{enumerate}
\item \label{p1} \emph{Subsolution in $\Omega$}: We have that 
\begin{itemize}
\item[]$-\Delta u_i  \le 0$ in distributional sense in $\Omega$, if $\mf{u}$ is a solution to problem (A),
\item[]$-\Delta u_i  \le \lambda_1(\omega_1)u_i$ in distributional sense in $\Omega$, if $\mf{u}$ is a solution to problem (B).
\end{itemize}
\item \label{p2} \emph{Solution in $S_i$:} We have that  
\begin{itemize}
\item[]$-\Delta u_i=0$ in $\textrm{int}(S_i)$, if $\mf{u}$ is a solution to problem (A),
\item[]$-\Delta u_i=\lambda_1(\omega_i)$ in $\textrm{int}(S_i)$, if $\mf{u}$ is a solution to problem (B).
\end{itemize}
\item \label{p3} \emph{Exterior sphere condition for the positivity sets:} $S_i$ satisfies the \emph{$1$-uniform exterior sphere condition in $\Omega$}, in the following sense: for every $x_0 \in \pa S_i \cap \Omega$ there exists a ball $B$ with radius $1$ which is exterior to $S_i$ and tangent to $S_i$ at $x_0$, i.e.
\[
S_i \cap B = \emptyset \quad \text{and} \quad x_0 \in \overline{S_i} \cap \overline{B}.
\]
Moreover, in $B \cap B_1(x_0)$ we have $u_j \equiv 0$ for every $j =1,\dots,k$ (including $j=i$).
\item \label{p4} \emph{Lipschitz continuity:} $u_i$ is Lipschitz continuous in $\Omega$, and in particular $S_i$ is an open set, for every $i$.
\item \label{p6} \emph{Lebesgue measure of the free-boundary:} the free-boundary $\pa \{u_i>0\}$ has zero Lebesgue measure, and its Hausdorff dimension is strictly smaller than $N$.
\item \label{p5} \emph{Exact distance between the supports:} for every $x_0 \in \pa S_i \cap \Omega$ there exists $j \neq i$ such that 
\[
\overline{B_1(x_0)} \cap \pa \, \supp u_j  \neq \emptyset.
\] 
\end{enumerate}
\end{theorem}

Notice that, if $y_0\in \pa S_j$ is such that  $|x_0-y_0|=1$, then $B_1(y_0)$ is an exterior sphere to $S_i$ at $x_0$. Moreover, by the Hopf lemma, the interior Lipschitz regularity is optimal.

Regarding the regularity of a vector of eigenfunctions $\mf{u}$ of problem (B), if we ask that $\Omega$ satisfies the exterior sphere condition,  then we have actually a stronger statement.

\begin{theorem}\label{thm: improved reg}
Let $\bu$ be a vector of first eigenfunctions associated to an optimal partition $(\omega_1,\ldots, \omega_k)$ of \eqref{eqn part dist}. Assume that $\Omega$ satisfies the exterior sphere condition with radius $r>0$. Then $\bu$ is globally Lipschitz continuous in $\overline{\Omega}$. 
\end{theorem}

Next, we establish a relation involving the normal derivatives of two ``adjacent components" on the regular part of the free boundary. 

In what follows, for each $i$, $\nu_i(x)$ will denote the exterior normal at a point $x\in \partial S_i$ (at points where such a normal vector does exist).

\medskip 

\noindent \textbf{Assumptions.}  Let $x_0 \in \pa S_i \cap \Omega$, and let us assume that $\Gamma_i^R:=\partial S_i \cap B_R(x_0)$ is a smooth hypersurface, for some $R>0$. By the $1$-uniform exterior sphere condition, we know that the principal curvatures of $\pa S_i$ in $x_0$, denoted by $\chi_{h}^{i}(x_0)$, $h=1,\dots,N-1$, are smaller than or equal to $1$ (where we agree that outward is the positive direction). We further suppose that the strict inequality holds, that is there exists $\delta >0$ such that 
\begin{equation}\label{eq:principalcurvatures}
\chi_1^{i}(x_0),\ldots, \chi_{N-1}^{i}(x_0) \le 1-\delta.
\end{equation}
We know that there exists $j \neq i$ and $y_0 \in \pa \, \supp u_j$ such that $|x_0-y_0|=1$. 

\begin{theorem}\label{thm:curvaturerelations}
Let $\bu=(u_1,\ldots, u_k)$ be either any minimizer of $J_\infty$ in $H_\infty$, or a vector of first eigenfunctions associated to an optimal partition $(\omega_1,\ldots, \omega_k)$ of \eqref{eqn part dist}. Under the previous assumptions and notations, we have that $y_0 = x_0 + \nu_i(x_0)$ is the unique point in $\bigcup_{k \neq i} \pa \, \supp u_k$ at distance $1$ from $x_0$. If $y_0 \in \pa\, \supp u_j \cap \Omega$, then $\pa\, \supp u_j$ is also smooth around $y_0$, and 
\begin{equation}\label{free-bound cond}
\frac{(\partial_\nu u_i(x_0))^2}{(\partial_\nu u_j(y_0))^2}= \left\{
\begin{array}{cl}
\displaystyle \mathop{\prod_{h=1}^{N-1} }_{\chi_h^{i}(x_0)\neq 0} \left|\frac{\chi_h^i(x_0)}{\chi_h^j(y_0)}\right| & \text{ if } \chi_h^i(x_0)\neq 0 \text{ for some $h$},\\
1  	& \text{ if $\chi_h^i(x_0)=0$  for all } h=1,\ldots, N-1.
\end{array}\right.
\end{equation}
\end{theorem}

We stress that, since the sets $S_i$ and $S_j$ are at distance $1$ from each other and \eqref{eq:principalcurvatures} holds, $\chi_h^i(x_0) \neq 0$ if and only if $\chi_h^j(y_0) \neq 0$, and hence the term on the right hand side is always well defined.

The proof of Theorem \ref{thm:curvaturerelations} is based on the introduction of a family of domain variations for the minimizer $\mf{u}$. As we shall see, the possibility of producing admissible domain variations, preserving the constraint on the distance of the supports in $H_\infty$, presents major difficulties. At the moment, we can only overcome such obstructions and produce more or less explicit variations supposing that $\pa S_i$ is locally regular. This is the main problem when trying to study the regularity of the free boundary. Regarding this point, we mention that the proofs of all our results (and also of those in \cite{CaPaQu}, in a nonvariational case) are completely different with respect to the analogue counterpart in problems with point-wise interaction. Indeed, all the local techniques, such as blow-up analysis and monotonicity formulae, cannot be straightforwardly adapted when dealing with long-range interaction; the reason is that the interface between different positivity sets $\{u_i>0\}$ and $\{u_j >0\}$ with $i \neq j$ is now a strip of width at least $1$, and hence with a standard blow-up one cannot catch the interaction on the free-boundary at the limit. 

We also mention that the validity of a uniform exterior sphere condition does not directly imply any extra regularity for $\pa S_i$: if we could show that $\pa S_i$ is a set with positive reach (see \cite{Fed}), then we could argue as in \cite[Corollary 6.3]{CaPaQu} and prove at least that the Hausdorff dimension of $\pa S_i$ is $N-1$ (see also \cite[Theorem 4.2]{ColMar} for a different proof of this fact), but on the other hand sets enjoying the uniform exterior sphere condition are not necessarily of positive reach, as shown in \cite[Section 2]{Nouretal.}.

\begin{remark}
A very interesting feature of Theorem \ref{thm:curvaturerelations} stays in the fact that it reveals a deep difference between segregation models with point-wise interaction, and with long-range interaction. To explain this difference, let us consider a sequence $\{\mf{u}_\beta\}$ of solutions to \eqref{syst p}, with $p=1$ and $\beta \to +\infty$. This is the setting studied in \cite{CaPaQu}. In \cite[Theorem 9.1]{CaPaQu}, the authors derive a free-boundary condition analogous to \eqref{free-bound cond} for the limit configurations in case $p=1$, but in their situation, the left hand side is replaced by the ratio between the normal derivatives, $\partial_\nu u_i(x_0) / \partial_\nu u_j(y_0)$. This difference is in contrast with respect to segregation phenomena with point-wise interaction, where, as proved in \cite{tt}, limit configurations associated with 
\[
-\Delta u_i = -\beta u_i \sum_{j \neq i} u_j \quad \text{or} \quad  -\Delta u_i = -\beta u_i \sum_{j \neq i} u_j^2
\]
belong to the same functional class \cite{DaWaZh,tt}, and hence in particular satisfy the same free-boundary condition, that is $|\partial_\nu u_i(x_0)| = |\partial_\nu u_j(x_0)|$ on the regular part of the free boundary. A similar difference has been observed in \cite{TVZ_half,TVZ_s,VZ_frac} in the case of fractional operators, that is when the non-locality is in the differential operator. 

Finally, in comparison with the free boundary condition derived in \cite{CaPaQu}, it is worthwhile noticing that the analogue of \eqref{free-bound cond} there involves the plain quotient of the normal derivatives, while here we find the squared one.
\end{remark}

\begin{remark}
The previous result may fail if the right hand side in \eqref{eq:principalcurvatures} is replaced by the constant $1$. Indeed, if $\partial S_i \cap B_R(x_0)=\partial B_1(0)\cap B_R(x_0)$ for some $x_0\in \partial B_1(0)$ and $R>0$, and the set $S_i$ is contained in the exterior of $B_1(0)$, then $y_0=0$ is a cusp for $\pa S_j$. 
\end{remark}

\subsection{Structure of the paper}

We first treat problem (A). In Section \ref{sec existence} we prove Theorem \ref{thm: existence} for this problem, relating this segregation problem with a variational competition--diffusion of type \eqref{syst p}. Then some qualitative properties of any possible minimizer of problem (A) are shown in Section \ref{sec:Properties_of_minimizers}, where we prove Theorem \ref{thm: properties minimizers} for this problem. Section \ref{FreeBoundary} contains the proof of the free boundary condition contained in the statement of  Theorem \ref{thm:curvaturerelations} for problem (A).

The analogous statements for problem (B) -- existence and properties of minimizers, and free boundary condition -- are proved in Section \ref{sec:ProblemB}. 

Finally, in Appendix \ref{appendix} we state and prove an Hadamard's type formula which we need along this paper.

\section{Existence of a minimizer for Problem (A)}\label{sec existence}
In this section we prove Theorem \ref{thm: existence}. To this purpose, we introduce a competition parameter $\beta>0$ which allows us to remove the segregation constraint. To be precise, 
let 
\begin{equation}\label{def H}
H=\{\bu\in H^1(\Omega_1,\R^k):\ u_i=f_i \text{ a.e. in } \Omega_1\setminus \Omega\} \supset H_\infty,
\end{equation}
and let $\beta>0$. We consider the minimization of the functional 
\[
J_\beta(\bu)=\sum_{i=1}^k \int_\Omega |\nabla u_i|^2 +\sum_{1\le i<j\le k} \iint_{\Omega_1 \times \Omega_1} \beta\, \mathds{1}_{B_1}(x-y)u_i^2(x) u_j^2(y)\, dx\,dy
\]
in the set $H$. With respect to the search of a minimizer for $\inf_{H_\infty} J_\infty$, the advantage stays in the fact that we can get rid of the infinite dimensional constraint $\dist(\supp u_i,\supp u_j)\ge 1$ for $i\neq j$, and we can easily show that a minimizer for $J_\beta$ in $H$ does exists, and satisfies an Euler-Lagrange equation of type \eqref{syst p} with $p=2$. This allows us to obtain Theorem \ref{thm: existence} as a direct corollary of the following statement:

\begin{theorem}\label{prop:existence}
For every $\beta>0$, there exists a minimizer $\mf{u}_\beta=(u_{1,\beta},\ldots, u_{k,\beta})$ for $\inf_H J_\beta$, which is a solution of
\begin{equation}\label{eq:system}
\begin{cases}
-\Delta u_i=-\beta u_i \sum_{j\neq i} (\mathds{1}_{B_1} \star u_j^2) & \text{in $\Omega$}\\
u_i > 0 & \text{in $\Omega$} \\
u_i=f_i & \text{in  $\Omega_1\setminus \Omega$}.
\end{cases}
\end{equation}
The family $\{\bu_\beta: \beta >0\}$ is uniformly bounded in $H^1(\Omega_1,\R^k)\cap L^\infty(\Omega_1)$, and there exists $\mf{u}=(u_1,\ldots, u_k) \in H$ such that:
\begin{enumerate}
\item $\bu_\beta \to \bu$ strongly in $H^1(\Omega)$ as $\beta\to +\infty$, up to a subsequence;
\item $\dist(\supp u_i,\supp u_j)\ge 1$ for every $i\neq j$, so that $\mf{u} \in H_\infty$;
\item for every $i\neq j$, 
\[
\lim_{\beta\to +\infty} \iint_{\Omega_1\times \Omega_1} \mathds{1}_{B_1}(x-y)u_{i,\beta}^2(x)u_{j,\beta}^2(y) \, dx\,dy=0
\]
\item $\bu$ is a minimizer for $\inf_{H_\infty} J_\infty$. In particular, $\mf{u}$ is a solution to problem (A).
\end{enumerate}
\end{theorem}

\begin{remark}
Without any additional complication, we can replace in the previous theorem the indicator function $\mathds{1}_{B_1}$ with a more general function $V\in L^\infty(\R^N)$ satisfying $V>0$ a.e. in $B_1$, $V=0$ a.e. on $\R^N\setminus \overline{B_1}$. 
\end{remark}

The proof of Theorem \ref{prop:existence} is the object of the rest of the section. Before proceeding, we observe that, by the definition of support given in \cite[Proposition 4.17]{Brezis_book_2011}, the set $H_\infty$ can be defined in the following equivalent way:
\[
\begin{split}
H_\infty &=\left\{\mf{u}\in H:\ \iint_{\Omega_1 \times \Omega_1} \mathds{1}_{B_1}(x-y) u_i^2(x)u_j^2(y) \, dx\,dy=0\ \forall i\neq j \right\}
\end{split}
\]
(see the proof of Lemma \ref{lem: harmonicity} below for more details).

\begin{remark}
Here it is worth to stress that we consider the functions $u_i$ as defined in $\Omega_1$, and hence the supports have to be considered in this set (and not only in $\Omega$).
\end{remark}


\begin{proof}[Proof of Theorem \ref{prop:existence}]
The existence of a minimizer $\mf{u}_\beta$ follows by the direct method of the calculus of variations, and the fact that minimizers solve \eqref{eq:system} is straightforward. Observe that $f_i\ge 0$, hence the minimizers are positive in $\Omega$, by the strong maximum principle.

For the uniform $L^\infty$ estimate, since $u_{i,\beta}>0$ is subharmonic in $\Omega$ for every $i=1,\dots,k$, by the maximum principle we have $\|u_{i,\beta}\|_{L^\infty(\Omega)} \le \|f_i\|_{L^\infty(\partial \Omega)}$. Let us set
\[
c_\beta:= \inf_H J_\beta \quad \text{and} \quad c_\infty:= \inf_{H_\infty} J_\infty.
\]
We observe that, since $J_\beta(\mf{v}) = J_\infty(\mf{v})$ for every $\mf{v} \in H_\infty$, we have $c_\beta \le c_\infty$. Then, by the minimality of $\mf{u}_\beta$, for every $\beta>0$ we have $J_\beta(\mf{u}_\beta) \le c_\infty$.
Since moreover $u_{i,\beta}\equiv f_i$ in $\Omega_1 \setminus \Omega$, the uniform $H^1(\Omega_1,\R^k)$ boundedness of $\{\mf{u}_\beta\}$ follows. Hence, up to a subsequence, $\mf{u}_\beta \rightharpoonup \mf{u}$ weakly in $H^1(\Omega_1,\R^k)$ and a.e. in $\Omega$. Moreover
\[
\lim_{\beta\to +\infty} \iint_{\Omega_1 \times \Omega_1} \mathds{1}_{B_1}(x-y)u_i^2(x) u_j^2(y)\, dx\, dy=0 \qquad \forall {i\neq j}
\]
and by the Fatou lemma we have
\[
0 \le \iint_{\Omega_1 \times \Omega_1} \mathds{1}_{B_1}(x-y) u_{i}^2(x)  u_{j}^2(y)  \, dx \,dy \le \liminf_{\beta \to +\infty} \iint_{\Omega_1 \times \Omega_1} \mathds{1}_{B_1}(x-y) u_{i,\beta}^2(x)  u_{j,\beta}^2(y)  \, dx \,dy= 0
\]
for every $i \neq j$. This in particular proves point (2) in the thesis and implies that $\mf{u} \in H_\infty$, defined in \eqref{def H infty}. 

 On the other hand, by the the minimality of $\mf{u}_\beta$ and weak convergence,
\begin{align*}
c_\infty & \le J_\infty(\mf{u}) = \sum_{i=1}^k \int_{\Omega} |\nabla u_{i}|^2 \le \liminf_{\beta \to \infty} \sum_{i=1}^k \int_{\Omega} |\nabla u_{i,\beta}|^2\\
&  \le \limsup_{\beta \to \infty} \sum_{i=1}^k \int_{\Omega} |\nabla u_{i,\beta}|^2 
 \le \limsup_{\beta \to \infty} J_\beta(\mf{u}_\beta) =\limsup_{\beta \to \infty} c_\beta\le c_\infty.
\end{align*}
This means that all the previous inequalities are indeed equalities, and in particular: 
\begin{itemize}
\item we have convergence $\|\nabla u_{i,\beta}\|_{L^2(\Omega)} \to \|\nabla u_i\|_{L^2(\Omega)}$,
which together with the weak convergence ensures that $\mf{u}_\beta \to \mf{u}$ strongly in $H^1(\Omega,\R^k)$ (recall that $\Omega$ is bounded);
\item 
point (3) of the thesis holds;
\item we have $c_\infty=J_\infty(\mf{u})$, which proves the minimality of $\mf{u}\in H_\infty$.  \qedhere
\end{itemize}
\end{proof}

\section{Properties of minimizers for problem (A)}\label{sec:Properties_of_minimizers}

This section is devoted to the proof of Theorem \ref{thm: properties minimizers} for the solutions of problem (A). Let then $\mf{u}$ be a minimizer for $\inf_{H_\infty} J_\infty$. 
Theorem \ref{thm: existence} (see also Theorem \ref{prop:existence}) does not give any information about the continuity of $u_i$, and in particular we do not know if the sets $S_i=\{x\in \Omega: u_i(x)>0\}$ are open. On the other hand it is reasonable to work at a first stage with the functions
\[
\Phi_i:\Omega \to \R,\qquad \Phi_i(x):=\int_{B_1(x)} u_i^2(y)\, dy,
\]
which are clearly continuous due to the Lebesgue dominated convergence theorem.

Let us consider the open sets
\[
C_i=\Omega \cap \left(\bigcup_{y\in \{\Phi_i=0\}} B_1(y)\right),\qquad D_i:= \textrm{int}\left( \Omega \setminus C_i\right),
\]
for $i=1,\ldots, k$, so that
\[
\Omega= C_i \cup D_i \cup (\partial D_i\cap \Omega),\qquad \text{ and } \qquad \partial D_i\cap \Omega=\partial C_i\cap \Omega.
\]
Observe that, by the definition of $\Phi_i$, we have $u_i=0$ a.e. in $C_i$. Moreover
\[
D_i = \{x \in \Omega: \dist(x,\{\Phi_i = 0\}) >1\} \subset \{\Phi_i>0\}.
\]
The strategy of the proof of Theorem \ref{thm: properties minimizers} can be summarized as follows:
\begin{itemize}
\item At first, we prove some simple properties of the set $D_i$ and of the restriction of $u$ on $D_i$.
\item In particular, we show that $S_i$ is the union of connected components of $D_i$, so that the regularity of $u_i$ in $\Omega$ is reduced to the regularity of $u_i$ on $\pa D_i$.
\item Using the basic properties of $D_i$, we show that $u_i$ is locally Lipschitz continuous across $\pa D_i$, and hence in $\Omega$. It follows in particular that $S_i$ is open, and directly inherits from $D_i$ properties \eqref{p3} and \eqref{p6} in Theorem \ref{thm: properties minimizers}. Moreover, points \eqref{p1} and \eqref{p2} holds.
\item As a last step, we prove point \eqref{p5} by using the minimality of $\mf{u}$.
\end{itemize}


\begin{lemma}\label{lem: harmonicity}
The function $u_i$ is harmonic in $D_i$. In particular, if $\tilde D_i$ is any connected component of $D_i$, then either $u_i\equiv 0$ or $u_i>0$ in $\tilde D_i$. 
\end{lemma}


\begin{proof}
The set $D_i$ is open. If we know that $\dist(D_i, \supp u_j) \ge 1$, then we can consider any $\phi\in C^\infty_c(D_i)$ and observe that, by the minimality of $\mf{u}$ for $J_\infty$ on the set $H_\infty$, the function
\[
f(\eps):=J_\infty(u_1,\ldots, u_{i-1}, u_i+\eps \phi, u_{i+1},\ldots, u_k)
\]
has a minimum at $\eps=0$. This implies that $u_i$ is harmonic in $D_i$, and all the other conclusions follow immediately. Therefore, in what follows we have to show that 
\begin{equation}\label{D supp}
\dist(D_i, \supp u_j) \ge 1 \qquad \forall j \neq i.
\end{equation}
By definition of $H_\infty$  we have $u_i^2(x)u_j^2(y)\mathds{1}_{B_1}(x-y)=0$ for a.e. $x,y\in \Omega_1$,
that is
\[
u_i^2(x)u_j^2(y)=0 \quad \text{ for a.e. } x,y\in \Omega_1,\ |x-y|<1.
\]
As a consequence, $u_j (x)\Phi_i(x) = 0$ for a.e. $x\in \Omega$ and every $j \neq i$. In particular, this implies that 
\begin{equation}\label{2353}
\{\Phi_i>0\} \subset \left( \Omega \setminus \supp u_j\right).
\end{equation}
Let $x_0 \in D_i$. Then by definition of $D_i$, $\dist(x_0,\{\Phi_i=0\}) > 1$, and hence $B_1(x_0) \subset \{\Phi_i>0\}$. But then, due to \eqref{2353}, and since $x_0$ has been arbitrarily chosen, we deduce that \eqref{D supp} holds.
\end{proof}

Let $A_i$ be the union of the connected components of $D_i$ on which $u_i>0$, and let $N_i$ be the union of those on which $u_i \equiv 0$, so that $D_i=A_i\cup N_i$. We know that $u_i$ is positive and harmonic in $A_i$, while $u_i =0$ a.e. in $N_i \cup C_i$. Since $A_i$, $N_i$ and $C_i$ are open, this means that (if necessary replacing $u_i$ with a different representative in its same equivalence class) $u_i$ is continuous in $A_i$, $N_i$, and $C_i$. 
To discuss the continuity of $u_i$ in $\Omega$, we have to derive some properties of the boundary $\pa D_i \cap \Omega= (\pa A_i \cup \pa N_i)\cap \Omega = \pa C_i \cap \Omega$. In the next lemma we show that $D_i$ satisfies a uniform exterior sphere condition.

\begin{lemma}\label{lem: ext sphere}
For each $i$, the set $D_i$ satisfies the $1$-uniform exterior sphere condition in $\Omega$, in the following sense: for every $x_0 \in \pa D_i \cap \Omega$ there exists a ball $B$ of radius $1$ such that 
\[
D_i \cap B = \emptyset \quad \text{and} \quad x_0 \in \overline{D_i} \cap \overline{B}.
\]
Moreover, in $B$ we have $u_i \equiv 0$.
\end{lemma}

\begin{proof}
This comes directly from the definitions: we have  
\[
\partial D_i \cap \Omega = \pa C_i \cap \Omega = \left\{x: \dist(x,\{\Phi_i=0\}) = 1\right\}\cap \Omega.
\]
Thus, given $x\in  \pa D_i\cap \Omega$, there exists $y\in \partial B_1(x)$ with $\Phi_i(y)=0$. The ball $B_1(y)$ is the desired exterior tangent ball, since $B_1(y)\subset C_i$, and hence $B_1(y)\cap D_i=\emptyset$.
\end{proof}

The exterior sphere condition permits to deduce that $\pa D_i$ has zero Lebesgue measure.

\begin{lemma}\label{lem: porosity}
The boundary $\pa D_i$ is a porous set, and in particular it has $0$ Lebesgue measure and $\textrm{dim}_{\mathfrak{H}} (\pa D_i) < N$.
\end{lemma}

For the definition of ``porosity", we refer to \cite[Section 3.2]{PeShUr}, while here and in what follows $\textrm{dim}_{\mathfrak{H}}$ denotes the Hausdorff dimension.

\begin{proof}
Since $\pa D_i \subset \Omega$ is bounded, to prove its porosity it is sufficient to show that there exists $\delta>0$ such that: for every ball $B_r(x_0)$ with $x_0 \in \pa D_i$, there exists $y \in B_r(x_0)$ with $B_{\delta r}(y) \subset B_r(x_0) \setminus \pa D_i$ (see \cite[Exercise 3.4]{PeShUr}). 

The existence of such $\delta = 1/2$ follows immediately by the exterior sphere condition: given $x_0 \in  \pa D_i$, there exists $z \in \Omega_1$ such that $B_1(z)$ is exterior to $D_i$. Let then $y$ be the point on the segment $x_0 z$ at distance $r/2$ from $D_i$. The ball $B_{r/2}(y)$ is contained both in $\Omega_1 \setminus \pa D_i$ and in $B_r(x_0)$, and this proves that $\pa D_i$ is porous. The rest of the proof follows by \cite[Page 62]{PeShUr}.
\end{proof}

It is not difficult now to deduce that $u_i$ is continuous at every point of $\pa N_i$. Indeed, notice that $\pa N_i \subset \pa C_i$, and in both $N_i$ and $C_i$ we have $u_i \equiv 0$. Since $\pa N_i \subset \pa D_i$ has $0$ Lebesgue measure, we deduce that $u_i=0$ a.e. in $\overline{N_i} \cup C_i = \Omega \setminus \overline{A_i}$. That is, up to the choice of a different representative,  $u_i \equiv 0$ in $\Omega \setminus \overline{A_i}$, and hence it is real analytic therein. At this stage, it remains to discuss the continuity of $u_i$ on $\pa A_i$. This is the content of the forthcoming Corollary \ref{thm lip}, where we show that actually $\mf{u}$ is locally Lipschitz continuous in $\Omega$. We postpone the proof, proceeding here with the conclusion of Theorem \ref{thm: properties minimizers}. The continuity of $u_i$ implies in particular that $\{u_i>0\}$ is open for every $i$, so that $\{u_i>0\}=A_i$. Thus, Lemmas \ref{lem: harmonicity}-\ref{lem: porosity} establish the validity of points \eqref{p2} and \eqref{p6} in Theorem \ref{thm: properties minimizers}. The subharmonicity of $u_i$, point \eqref{p1}, follows from \eqref{p2}.

Regarding point \eqref{p3}, the existence of an exterior sphere $B$ of radius $1$ for $\{u_i>0\}$ at any boundary point $x_0$ comes directly from Lemma \ref{lem: ext sphere}. We also know that $u_i \equiv 0$ in $B$, and furthermore, by \eqref{D supp}, $B_1(x_0) \cap \supp u_j = \emptyset$ for every $j \neq i$. This proves the validity of \eqref{p3}.

It remains only to show that also point \eqref{p5} holds.

\begin{proof}[Proof of Theorem \ref{thm: properties minimizers}-\eqref{p5}]
This is a consequence of the minimality. Take $x_0\in \partial S_i\cap \Omega$ and assume, in view of a contradiction, that $\dist(x_0,\supp u_j)>1$ for some $x_0\in \partial S_i \cap \Omega$, for every $j\neq i$. Then there exists $\rho>0$ such that $B_\rho(x_0)\subset \Omega$ and
\begin{equation}\label{2551}
\dist(B_\rho(x_0),\supp u_j)> 1 \qquad \forall j \neq i.
\end{equation}
Let $v$ be the harmonic extension of $u_i$ in $B_\rho(x_0)$:
\[
\begin{cases}
\Delta v=0 & \text{in $B_\rho(x_0)$} \\
v=u_i & \text{on $\pa B_\rho(x_0)$}.
\end{cases}
\]
Since $u_i \not \equiv 0$ on $\pa B_\rho(x_0)$, we infer that $v>0$ in $B_\rho(x_0)$, and in particular $v \not \equiv u_i$ in $B_\rho(x_0)$. Let now $\tilde{\bu}$ be defined by
\[
\tilde u_i=\begin{cases} u_i & \text{ in } \Omega \setminus B_\rho(x_0)\\
v & \text{ in } B_\rho(x_0)
\end{cases},
\qquad \tilde u_j=u_j \quad \forall j\neq i.
\]
Due to \eqref{2551}, it belongs to $H_\infty$, so that by minimality $J_\infty(\mf{u}) \le J_\infty(\tilde{\mf{u}})$. On the other hand, by the definition of harmonic extension we have also $J_\infty(\tilde{\mf{u}}) < J_\infty(\mf{u})$ (the strict inequality comes from the fact that $v \not \equiv u_i$ in $B_\rho(x_0)$), a contradiction. 
\end{proof}

\begin{remark}
In \cite{CaPaQu}, the authors proved harmonicity, local Lipschitz continuity, and exterior sphere condition for limits of any sequence of solutions to \eqref{eq:system}. Nevertheless, the result here is not contained in \cite{CaPaQu}, since we establish harmonicity, Lipschitz continuity, and exterior sphere condition for any minimizer of $\inf_{H_\infty} J_\infty$, independently on wether it can be approximated with a sequence of solutions to \eqref{eq:system} or not. Also, it is worth to point out that the approach is completely different: while in \cite{CaPaQu} the authors proceed with careful uniform estimates for viscosity solution of \eqref{syst p}, here we use the variational structure of the limit problem. 
\end{remark}

\subsection{Lipschitz continuity of the minimizers}\label{sec: Lip}
In this subsection we show that the solutions of problem (A) are Lipschitz continuous inside $\Omega$, which is the highest regularity one can expect for the minimizers of $J_\infty$ (by the Hopf lemma). This is a consequence of the following general statement.

\begin{theorem}\label{thm: lip general}
Let $\Lambda$ be a domain of $\R^N$, and let $A \subset \Lambda$ be an open subset, satisfying the $r$-uniform exterior sphere condition in $\Lambda$: for any $x_0 \in \pa A \cap \Lambda$ there exists a ball $B$ with radius $r$ which is exterior to $A$ and tangent to $\pa A$ at $x_0$, i.e.
\[
A \cap B = \emptyset \quad \text{and} \quad x_0 \in \overline{A} \cap \overline{B}.
\]
Let $f \in L^\infty(\Lambda)$, and let $u \in H^1(\Lambda) \cap L^\infty(\Lambda)$ satisfy
\begin{equation}\label{eq u}
\begin{cases}
-\Delta u = f & \text{in $A$} \\
u = 0 & \text{a.e. in $\Lambda \setminus A$}
\end{cases}
\end{equation}
Then $u$ is locally Lipschitz continuous in $\Lambda$, and for every compact set $K \Subset \Lambda$ there exists a constant $C=C(r,N,K)>0$ such that 
\[
\|\nabla u\|_{L^\infty(K)} \le  C\left(  \|u\|_{L^\infty(\Lambda)} + \|f\|_{L^\infty(\Lambda)}\right).
\]
\end{theorem}

For the sake of generality, we required no sign condition on the function $u$, even though we will apply the result only to nonnegative solutions.

\begin{corollary}\label{thm lip}
Let $\mf{u}$ be any minimizer of $J_\infty$ in $H_\infty$. Then $\mf{u}$ is locally Lipschitz continuous in $\Omega$.
\end{corollary}

\begin{proof}
We apply Theorem \ref{thm: lip general} to the harmonic functions $u_i$ in $A:=A_i$, with $\Lambda:=\Omega$ and $r=1$.
\end{proof}

The proof of Theorem \ref{thm: lip general} is based upon a simple barrier argument. For any $R>0$, let us define
\[
w_R(x):= \frac{1}{2N}(R^2-|x|^2)^+  \quad \implies \quad \begin{cases}
-\Delta w_R =1 & \text{in $B_R$} \\
w_R=0 & \text{in $\R^N \setminus B_R$},
\end{cases}
\]
and let 
\begin{equation}\label{def w *}
w_R^*(x):= \left(\frac{R}{|x|}\right)^{N-2} w_R\left(\frac{R^2}{|x|^2}x\right) = \frac{R^N}{2N |x|^N} \left( |x|^2-R^2 \right)^+
\end{equation}
be its Kelvin transform with respect to the sphere of radius $R$. It is not difficult to check that 
\begin{equation}\label{delta w r}
-\Delta w_R^*(x) = -\left(\frac{R}{|x|}\right)^{N+2} \Delta w_R\left(\frac{R^2}{|x|^2}x\right) = \left(\frac{R}{|x|}\right)^{N+2}.
\end{equation}
With this preliminary observation, we can easily prove the following estimate:

\begin{lemma}
Let $x_0 \in \pa A \cap \Lambda$, and let $\rho>0$ be such that $B_\rho(x_0) \Subset \Lambda$. Under the assumptions of Theorem \ref{thm: lip general}, there exists a constant $C>0$ depending on the dimension $N$, on $r$ and on $\rho$, such that
\[
|u(x)| \le C \left( \|u\|_{L^\infty(\Lambda)} + \|f\|_{L^\infty(\Lambda)}\right) |x-x_0| \qquad \forall x \in B_{\rho}(x_0).
\]
\end{lemma}
\begin{proof}
Let $y_0 \in \R^N$ be the center of the exterior sphere in $x_0$: 
\[
A \cap B_r(y_0) = \emptyset \quad \text{and} \quad x_0 \in \overline{A} \cap \overline{B_r(y_0)}
\]
Let $z_0$ be the medium point on the segment $x_0 \, y_0$.
Up to a rigid motion, we can suppose that $z_0=0$ and that $x_0= (0',r/2)$, where $0'$ denotes the $0$ vector in $\R^{N-1}$. In this setting, we aim at proving that $u \le w_{r/2}^*$ in $B_{\rho}(x_0) \cap A$, with $w_{r/2}^*$ defined by \eqref{def w *}. Since $u=0$ a.e. in $\Omega \setminus A$, we have (in the sense of traces) that $u=0$ on $\pa A \cap \overline{B_{\rho}(x_0)}$. Moreover, since $B_{r/2}(z_0) \subset B_r(y_0)$, and $\pa B_{r/2}(z_0) \cap \pa B_r(y_0) =\{x_0\}$, there exists a value $\delta = \delta (r,\rho, N)$ (independent on the point $x_0$) such that $\dist(z_0, A \cap \pa B_{\rho}(x_0)) \ge \dist(z_0,\partial B_r(y_0)\cap \partial B_\rho(x_0))> r/2+\delta$. Hence
\[
\inf_{A \cap \pa B_{\rho}(x_0)} w_{r/2}^* \ge m( r,\rho, N )>0,
\]
and we can define
\[
\varphi(x):=\left( \frac{\|u\|_{L^\infty(\Omega)}}{m( r,\rho, N )}+ \left(\frac{2\rho}{r}\right)^{N+2}\|f\|_{L^\infty(\Lambda)} \right) w_{r/2}^*(x).
\]
It is now not difficult to check that
\[
\begin{cases}
-\Delta (\varphi-u) \ge 0 & \text{in $A \cap B_{\rho}(x_0)$} \\
(\varphi-u) \ge 0 & \text{on $\pa (A \cap B_{\rho}(x_0))$}
\end{cases}
\]
Indeed, in $A\cap B_\rho(x_0)$, by recalling  \eqref{delta w r},
\[
-\Delta u=f\leq \|f\|_{L^{\infty}(\Lambda)} \quad \text{ and } \quad  -\Delta w_{r/2}^*=\left(\frac{r}{2|x|}\right)^{N+2}\geq \left(\frac{r}{2\rho}\right)^{N+2}
\]
The boundary $\pa (A \cap B_{\rho}(x_0))$ splits into two parts. On the first part $\pa A \cap \overline{B_{\rho}(x_0)}$ we know that $u=0$ in the sense of traces, and since $\varphi \ge 0$ there, we have $\varphi-u\ge 0$ on $\pa A \cap \overline{B_{\rho}(x_0)}$ in the sense of traces. On the remaining part $A \cap \pa B_{\rho(x_0)}$, the function $u$ can be evaluated point-wisely, since in the interior of $A$ the function $u$ is of class $\mathcal{C}^{1,\alpha}$; therefore, it makes sense to write that $u(x) \le \|u\|_{L^\infty(\Omega)} \le \varphi$ for any $x \in A \cap \pa B_{\rho(x_0)}$. All together, we obtain that $u \le \varphi$ on $\pa (A \cap B_{\rho}(x_0))$ in the sense of traces.

In conclusion, we have $u \le \varphi$ in $A \cap B_{\rho}(x_0)$ by the maximum principle. Observing that 
\begin{align*}
\frac{r^N}{2^{N+1} N |x|^N} \left( |x|^2-\left(\frac{r}2\right)^2 \right) &= \frac{r^N}{2^{N+1}N |x|^N} \left(|x|+\frac{r}{2}\right)\left(|x|-\frac{r}{2}\right)\\
& \le \frac{r^N}{2^{N+1}N(r/2)^N}\left(\rho+\frac{r}{2}\right) \left(|x|-|x_0|\right)\le \frac{2^{N-1}(2\rho+r)}{N}|x-x_0|.
\end{align*}
for every $x \in B_{\rho}(x_0)$, we obtain the desired upper estimate for $u$. Arguing in the same way on $-u$, we obtain also the lower estimate, and the proof is complete. 
\end{proof}

As an immediate consequence:

\begin{corollary}\label{lem: 5 gen 1}
For every compact set $K \Subset \Lambda$ there exists $C=C(K,r,N)>0$ such that
\[
|u(x)| \le C \left( \|u\|_{L^\infty(\Lambda)} + \|f\|_{L^\infty(\Lambda)}\right) \dist(x,\pa A)
\]
whenever $x\in K$ with $\dist(x,\pa A) < \dist(K,\pa \Lambda)$.
\end{corollary}

\begin{proof}
Let $x\in A\cap K$ such that $\dist(x,\partial A)<\dist(K,\partial \Lambda)$. Then take $x_0\in \partial A$ such that $|x-x_0|=\dist(x,\partial A)$. Then we can apply the previous theorem to $B_{\dist(K,\pa \Lambda)/2}(x_0)$.
\end{proof}

We are ready to proceed with the:

\begin{proof}[Proof of Theorem \ref{thm: lip general}]
Recall that $-\Delta u=f$ in $A$, hence there the function $u$ is of class $\mathcal{C}^{1,\alpha}$. Since moreover $u \in H^1(\Lambda)$ and $\nabla u=0$ a.e. in $\Lambda \setminus A$,  it is sufficient to obtain a uniform estimate for $\nabla u$ in a neighborhood of $\partial A$ (and actually only in $A$). Notice that in $A$ it makes sense to consider point-wise values of the gradient of $u$. 

We use the notation $d_x:=\dist(x,\partial A)$, for every $x\in \Omega$. Take $x_0\in \partial A\cap \Lambda$ and let $\delta>0$ be small enough such that, considering the compact set
\[
\displaystyle K:= \overline{\bigcup_{x\in B_\delta(x_0)} B_{d_x}(x)}
\]
then
\[
\dist(x,\partial A)<\dist(K,\partial \Lambda)\qquad \forall x\in K.
\]
By Corollary \ref{lem: 5 gen 1}, there exists $C=C(K,N,r)>0$ such that
\begin{equation}
|u(x)| \le C\left(  \|u\|_{L^\infty(\Lambda)} + \|f\|_{L^\infty(\Lambda)}\right) d_x \qquad \forall x\in K.
\end{equation}
In particular, for every $x\in A\cap B_\delta(x_0)$, since $B_{d_{x}}(x)\subset K$, then
\begin{equation}\label{eq 5 gen 2}
\|u\|_{L^\infty(B_{d_x}(x))}\leq 2C\left(  \|u\|_{L^\infty(\Lambda)} + \|f\|_{L^\infty(\Lambda)}\right) d_x.
\end{equation}
Now, let 
\[
Q_x:= \left\{ y \in \R^N: |y_i-x_i|<\frac{d_x}{m}, \ i=1,\dots,N\right\},
\]
where $m>0$ is chosen so large that the cube $Q_x$ is contained in the ball $B_{d_x}(x)$ ($m>0$ is a universal constant, depending only on the dimension $N$). Since $B_{d_x}(x) \subset A$, then $-\Delta u=f$ in $B_{d_x}(x)$ and we can combine \eqref{eq 5 gen 2} with interior gradient estimates for the Poisson equation (see \cite[Formula 3.15)]{GiTr}), deducing that  
\[
|\nabla u(x)| \le \frac{Nm}{d_x} \sup_{\pa Q_x} |u| + \frac{ d_x}{2m} \sup_{Q_x} |f| \le C'\left(  \|u\|_{L^\infty(\Lambda)} + \|f\|_{L^\infty(\Lambda)}\right) \qquad \forall x\in A\cap B_\delta(x_0).
\qedhere \]
\end{proof}

\section{Free-boundary condition for problem (A)}\label{FreeBoundary}

In this section we prove Theorem \ref{thm:curvaturerelations}. We briefly recall the setting.

Let $x_0 \in \pa S_i \cap \Omega$, and let us assume that $\Gamma_i^R:=\partial S_i \cap B_R(x_0)$ is a smooth hypersurface, for some $R>0$. We suppose that, for a positive $\delta$, condition \eqref{eq:principalcurvatures} holds on $\Gamma_i^R$:
\[
\chi_1^i(x),\ldots, \chi_{N-1}^i(x) \le 1-\delta \qquad \forall x \in \Gamma_i^R,
\]
where $\chi_1^i, \dots, \chi_{N-1}^i$ denote the principal curvatures of $\pa S_i$. Without loss of generality, we can suppose that $\Gamma_i^R$ is a graph:
\[
\Gamma_i^R=\{(x',\psi(x')):\ x'\in B_R^{N-1}(x_0')\}, \quad \text{ and } \quad S_i\cap B_R(x_0)=\{(x',z)\in B_R(x_0):\ z\le \psi(x')\}
\] 
for a function $\psi:B_R^{N-1}(x_0')\to \R$, where $B_R^{N-1}(x_0')$ denotes the ball of radius $R$ in $\R^{N-1}$ centered at $x_0'=(x_0^1,\ldots, x_0^{N-1})$. We know from Theorem \ref{thm: properties minimizers}-(6) that there exists $j \neq i$ and $y_0 \in \pa \, \supp u_j$ such that $|x_0-y_0|=1$. 

The proof of Theorem \ref{thm:curvaturerelations} is divided into several steps. We start with the uniqueness and characterization of $y_0$.
\begin{lemma}\label{lem: uniqueness y0}
If $x \in \pa S_i \cap \Omega$ and $\pa S_i$ is smooth in a neighbourhood of $x$, then $y= x+ \nu_i(x)$ is the unique point in $\bigcup_{l \neq i} \pa \, \supp u_l$ at distance $1$ from $x$.
\end{lemma}
\begin{proof}
By Theorem \ref{thm: properties minimizers} (points (3) and (6)), we know that there exists a point $y \in \bigcup_{l \neq i} \pa \, \supp u_l$ such that
\[
|x-y|=1, \quad \text{and} \quad |x-z| \ge 1 \quad \text{for all }z \in \bigcup_{l \neq i} \pa \, \supp u_l.
\]
This means that $y-x \in Q:=\{v: \dist(x+v,S_i) = |v|\}$. By \cite[Theorem 4.8-(2)]{Fed}, $Q$ is a subset of the normal cone to $S_i$ in $x$, and since $\pa S_i$ is smooth in $x$, we deduce that $y-x= \nu_{i}(x)$. 
\end{proof}

The previous lemma implies that there exists a unique $j$ and a unique $y_0 \in \pa \, \supp u_j$ at distance $1$ from $x_0$.  In order to simplify the notation, let $i=1$ and $j=1$, and so $x_0\in \partial S_1\cap \Omega$, $y_0\in \partial \, \supp u_2$. Assume from now on that $y_0\in \Omega$, so that $y_0\in \partial S_2\cap \Omega$. We denote $\Gamma_{1}^R:=  \partial S_1\cap B_R(x_0)$ and $\Gamma_{2}^R:=\{x+\nu_1(x):\ x\in \Gamma_{1}^R\}$. Notice that by Lemma \ref{lem: uniqueness y0} and by continuity, we have that $y_0\in \Gamma_{2}^R\subset \partial S_2\cap \Omega$, where the last inclusion holds for sufficiently small $R>0$. 

\begin{lemma}
The set $\Gamma_{2}^R$ is a smooth hypersurface.
\end{lemma}
\begin{proof}
 
The set $\Gamma_2^R$ can be parametrized by $\Phi:B_R^{N-1}(x_0') \to \R^N$, 
\[
\Phi(x')= (x',\psi(x'))+\nu_1(x',\psi(x'))=\left(x'-\frac{\nabla \psi(x')}{\sqrt{1+|\nabla \psi(x')|^2}},\psi(x')+\frac{1}{\sqrt{1+|\nabla \psi(x')|^2}}\right),
\] 
and hence we need to prove that $D\Phi(x')$ has maximum rank. We have
\begin{equation}\label{eq:Gamma_2_hypersurface}
D \left(x'-\frac{\nabla \psi(x')}{\sqrt{1+|\nabla \psi(x')|^2}}\right)=\textrm{Id}_{N-1}-D\left(\frac{\nabla \psi(x')}{\sqrt{1+|\nabla \psi(x')|^2}}\right),
\end{equation}
where $\textrm{Id}_{N-1}$ denotes the identity in $\R^{N-1}$. Observe that $D\left(\nabla \psi(x')/\sqrt{1+|\nabla \psi(x')|^2}\right)$ is the curvature tensor of $\Gamma_1^R$ at $(x',\psi(x'))$ (see for instance \cite[p.356]{GiTr}). Assumption \eqref{eq:principalcurvatures} implies that all its eigenvalues are strictly smaller than one. Then the determinant of \eqref{eq:Gamma_2_hypersurface} does not vanish, and the result follows.
\end{proof}

Observe that, with the previous notations,
\begin{equation}\label{eq:relation_between_normals}
\nu_1(x)=-\nu_2(x+\nu_1(x))\ \forall x\in \Gamma_{1}^R\quad \text{ and } \quad \nu_2(x)=-\nu_1(x+\nu_2(x))\ \forall x\in \Gamma_{2}^R.
\end{equation}
Let $\eta\in \mathcal{C}^\infty_{\rm c}(B_R(x_0))$ be a nonnegative test function. We define two deformations, one acting on $S_1$, and the other on $S_2$. The first one, which deforms $S_1$, is a function denoted by $F_{1,\eps}: \R^N \to \R^N$, $\eps\in [0,\bar \eps)$, such that, 
\[
F_{1,\eps}(x)=
\begin{cases}
x & \text{ if } x\not\in B_R(x_0)\\
x+\eps \eta(x) \nu_1(x) & \text{ if } x\in \Gamma_{1}^R,
\end{cases}
\]
extended to the whole $\R^N$ in such a way that $(\eps,x)\in [0,\bar \eps) \times  \R^N \mapsto F_{1,\eps}(x)$ is of class $\mathcal{C}^1$, and $F_{1,0}(\cdot)=\textrm{Id}$. We denote 
\[
S_{1,\eps}:=F_{1,\eps}(S_1):=S_1\cup \{x+s \eta(x)\nu_1(x):\ x\in \Gamma_{1}^R,\ 0\le s<\eps\}
\] 
 and  
 \[
\Gamma_{1,\eps}^R:=F_{1,\eps}(\Gamma_{1}^R)=\{x+\eps \eta(x) \nu_1(x):\ x\in \Gamma_{1}^R\}.
\]

\begin{lemma}
The set $\Gamma_{1,\eps}^R$ is a smooth hypersurface. Moreover, if we denote its exterior normal at a point $x+\eps \eta(x) \nu_1(x)$ (for $x\in \Gamma_{1}^R$) by $\nu^\eps(x)$, then $\eps\mapsto \nu^\eps(x)$ is differentiable at $\eps=0$ and 
\begin{equation}\label{eq:derivative_normal_orthogonal}
\left. \frac{d}{d\eps} \nu^\eps(x)\right|_{\eps=0} \text{ is orthogonal to } \nu_1(x), \text{ for every } x\in \Gamma_{1}^R.
\end{equation}
\end{lemma}
\begin{proof}
By the smoothness of $\Gamma_{1}^{R}$ and of the perturbation $\eta$, it follows that $\nu^\eps$ is differentiable in $\eps$ for $\eps$ small.
By deriving the identity $|\nu^\eps(x)|^2=1$ in $\eps$ for each $x\in \Gamma_{1}^R$, we have $\frac{d}{d\eps} \nu^\eps(x)\cdot \nu^\eps(x)=0$. Since $\nu^0(x)=\nu_1(x)$, the statement \eqref{eq:derivative_normal_orthogonal} follows.
\end{proof}

Now we consider an open neighbourhood $B_{y_0}^R$ of $y_0$ such that $B_{y_0}^R\cap \partial S_2=\Gamma_{2}^R$ and $\dist(B_{y_0}^R,\partial \Omega)>0$. In order to deform $S_2$, we take  $F_{2,\eps}:\R^N \to \R^N$, $\eps\in [0,\bar \eps)$, such that
\[
F_{2,\eps}(y)=
\begin{cases}
y & \text{ if } y\not\in B_{y_0}^R,\\
x + \eps \eta(x) \nu_1(x)+ \nu^\eps(x), & \text{ if } x=y+\nu_2(y),\ y\in \Gamma_{2}^R\\
\end{cases}
\]
extended to the whole $\R^N$ in such a way that $(\eps,x) \in [0,\bar \eps)\times \R^N \mapsto F_{2,\eps}(x)$ is of class $\mathcal{C}^1$, and $F_{2,0}(\cdot)=\textrm{Id}$. Define
\[
S_{2,\eps}:=F_{2,\eps}(S_2):=S_2\backslash \{ x + s \eta(x) \nu_1(x)+ \nu^\eps(x):\ x\in \Gamma_{1}^R,\ 0\le s<\eps\} 
\]
and 
\[
\Gamma_{2,\eps}^R:= F_{2,\eps}(\Gamma_{2}^R)=\{x + \eps \eta(x) \nu_1(x)+ \nu^s(x):\ x\in \Gamma_{1}^R \}.
\]
Notice that, since $\eta \ge 0$, we have $\Gamma_{2,\eps}^R \subset \overline{S_2}$ for every $\eps > 0$.

\begin{figure}
\centering
\includegraphics[width=\linewidth]{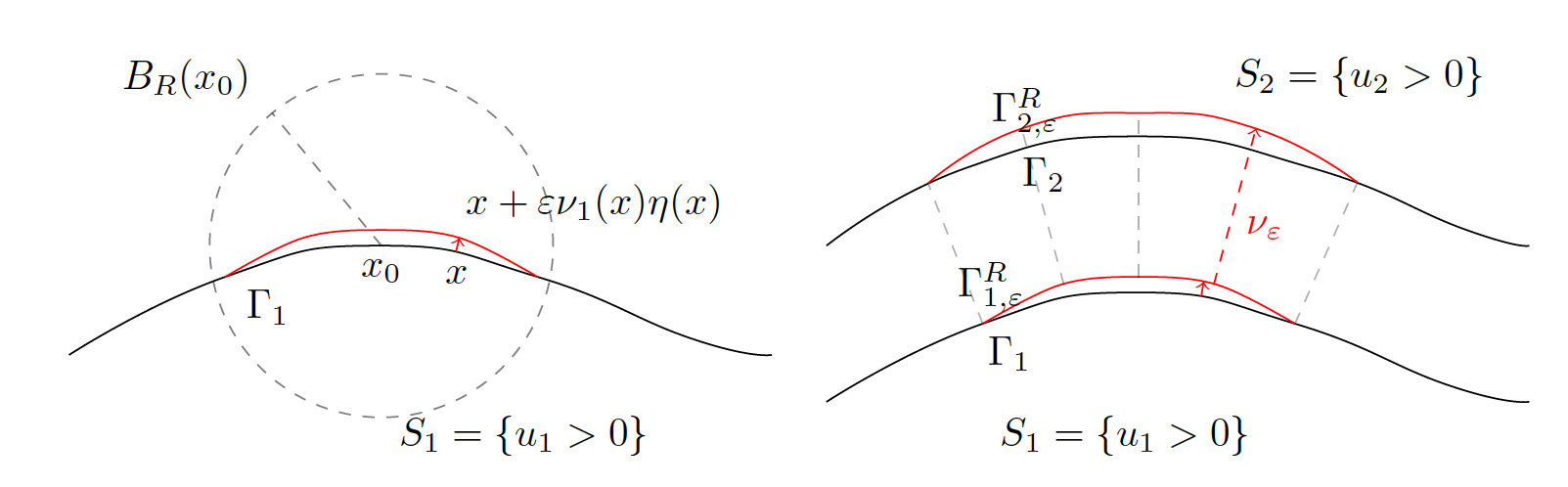}
\caption{The picture on the left represents the deformation acting on $S_1$. The picture on the right represents the deformation acting on $S_2$.}
\end{figure}

\begin{remark}\label{rem: gamma gamma eps}
We observe that the map $x \in \Gamma_1^R \mapsto x_\eps:= x+ \eps \eta(x) \nu_1(x) \in \Gamma_{1,\eps}^R$ is a diffeomorphism for $\eps>0$ small enough. 
For this reason, we can see the normal $\nu^\eps$ as defined on $\Gamma_{1,\eps}^R$, and use the notation 
\[
\nu^\eps(x_\eps):= \nu^\eps(x) \quad \iff x_\eps =x + \eps \eta(x) \nu(x).
\]  
\end{remark}

The crucial point in our argument is the following:

\begin{lemma}\label{lemma:dist_perturbedsupp}
We have $\dist(S_{1,\eps},S_{2,\eps})\ge 1$. Moreover, $\dist(S_{i,\eps},S_j)\ge 1$ for every $i\in \{1,2\}$, $j\neq 1,2$.
\end{lemma}

For the proof we will need the following elementary fact. 
\begin{lemma}\label{lem: curvatura piana}
Let $(x_1,y_1)$, $(x_2,y_2)$, two points on the lower semi-circle $\pa B_1^-:= \{x^2 + y^2 = 1, y <0\}$ in $\R^2$. Let $\gamma$ be the graph of a $\mathcal{C}^2$ function $f:[x_1,x_2] \to \R$, and let us suppose that:
\begin{itemize}
\item the curvature of $\gamma$ is strictly smaller than $1$;
\item $f(x_1)= y_1$, i.e. $(x_1,y_1)$ is the initial point of $\gamma$;
\item there exists $\rho>0$ such that $f(t) \le -\sqrt{1-t^2}$ for $t \in (x_1,x_1+\rho)$.
\end{itemize}
Then $f(x_2) < y_2$, i.e. $\gamma$ cannot contain any other point on $\pa B_1$.
\end{lemma}
\begin{proof}
In terms of $f$, the curvature of $\gamma$ is defined by 
\[
k(t):= \frac{f''(t)}{\left(1+(f'(t))^2\right)^{3/2}}.
\]
Thus, by assumption:
\[
f''(t) < \left(1+(f'(t))^2\right)^{3/2} \text{in $[x_1,x_2]$},\qquad  f'(x_1) \le \frac{x_1}{\sqrt{1-x_1^2}}\ \text{ and }\ f(x_1) = y_1.
\]
Recalling that $v(t)=-\sqrt{1-t^2}$ solves $v''=(1+(v')^2)^{3/2}$, the thesis follows by a comparison argument for solutions to ODEs.\end{proof}

\begin{proof}[Proof of Lemma \ref{lemma:dist_perturbedsupp}]
The second statement of the lemma comes from the fact that $\dist(S_i,S_j)\geq 1$ and $\dist(\Gamma_{i}^R,S_j)>1$ for $i=1,2$ and $j>2$. As for the first statement, observe that it is enough to show that
\[
\dist(\partial S_{1,\eps}\cap \Omega, \partial S_{2,\eps}\cap \Omega)\ge 1.
\]
By construction, $\partial S_{i,\eps}\setminus \Gamma^R_{i,\eps}= \partial S_i\setminus \Gamma_i^R$ for $i=1,2$, and since $\dist(\partial \, \supp u_1,\partial \, \supp u_2)=1$, then 
\[
\dist ( \partial S_{1,\eps}\setminus \Gamma^R_{1,\eps},\partial S_{2,\eps}\setminus \Gamma^R_{2,\eps})\ge 1.
\] 
Since every point in $\Gamma_i^R$ admits a unique point on $\partial S_j$ at distance exactly one, we have that $\dist(\Gamma_i^R,\partial S_j\setminus \Gamma_j^R)>1$ for every $i\neq j$, $i,j\in \{1,2\}$. Thus, by the continuity of the deformations $F_{1,\eps},F_{2,\eps}$, 
\[
\dist(\Gamma_{i,\eps}^R,\partial S_{j,\eps}\setminus \Gamma^R_{j,\eps})>1\qquad \forall i\neq j,\ i,j\in \{1,2\}.
\]
It remains to show that
\[
\dist(\Gamma_{1,\eps}^R,\Gamma_{2,\eps}^R)=1.
\]
This follows from the following property (we use the notation introduced in Remark \ref{rem: gamma gamma eps}):
\begin{itemize}
\item[(C)] there exists $\eps>0$ small enough such that any point in $y \in S_{1,\eps}^c$ such that $\dist(y, \Gamma^R_{1,\eps}) = 1$ has unique projection at minimal distance onto $S_{1,\eps}$, this projection lies in $\Gamma_{1,\eps}^R$, and moreover $y= x_\eps +  \nu^\eps(x_\eps)$ for some $x_\eps \in \Gamma_{1,\eps}^R$. 
\end{itemize}
Indeed, (C) implies, by definition of $\Gamma_{2,\eps}^R$, that
\begin{align*}
 \{y \in S_{1,\eps}^c: \dist(y, \Gamma^R_{1,\eps}) = 1\} &= \{y \in \overline{S_2}: \dist(y, \Gamma_{1,\eps}^R) = 1\} \\
&= \{y \in \overline{S_2}: y= x_\eps +  \nu^\eps(x_\eps), \ x_\eps \in \Gamma_{1,\eps}^R\} = \Gamma_{2,\eps}^R,
\end{align*}
and completes the proof.

\smallbreak

Let us now prove property (C). That any point at minimal distance from $S_{1,\eps}$ stays on $\Gamma_{1,\eps}^R$ is a consequence of Lemma \ref{lem: uniqueness y0} for $\eps=0$; the case $\eps>0$ small follows by continuity of $F_{1,\eps}$, and recalling that $\eta$ has compact support. Take $y \in \overline{S_2} \cap \{ \dist(z, \Gamma_{1,\eps}^R) = 1\}$. To prove the uniqueness of the projection, suppose by contradiction that there exist two points $x_1$ and $x_2$ in $\Gamma_{1,\eps}^R$ such that $|x_1-y| = |x_2-y|=1$. Since our argument is local in nature, it is not restrictive to suppose that we chose $R<1/2$ from the beginning, and hence in particular $|x_1-x_2|<1$.

Let $\Pi$ be the plane containing $x_1,x_2$ and $y$, and let $\gamma$ be the arc of the curve $\Gamma_{1,\eps}^R \cap \Pi$ connecting $x_1$ and $x_2$. The basic idea which we develop in what follows is that the existence of both $x_1$ and $x_2$ is forbidden by the fact that, thanks to \eqref{eq:principalcurvatures}, the curvature at every point of $\gamma$ is smaller than $1$. 

Since $\Gamma_{1,\eps}^R$ is a graph of a function of $x_N$,
\begin{equation}\label{pt1}
\text{also $\gamma$ can be seen as the graph of a function of $x_N$ for $\eps$ small enough.}
\end{equation} 
Also, since the principal curvatures of $\pa S_1$ are all smaller than $1-\delta$ on $\Gamma_{1}^R$, for $\eps$ small enough the principal curvatures of $\pa S_{1,\eps}$ are all smaller than $1-\delta/2)$ on $\Gamma_{1,\eps}^R$. Combining this with the fact that $x_1$ is a projection of $y$ onto $S_{1,\eps}$, it follows the existence of $r>0$ small (possibly depending on $\eps$) such that 
\begin{equation}\label{pt2}
B_r(x_1) \cap \Gamma_{1,\eps}^R \cap B_{1}(y) = \{x_1\}.
\end{equation}
Moreover, 
\begin{equation}\label{pt3}
\text{the (planar) curvature of $\gamma$ is also smaller than $1-\delta/2$.}
\end{equation}

Collecting together \eqref{pt1}, \eqref{pt2}, \eqref{pt3}, we are in position to apply \footnote{after a translation and a possible rotation} Lemma \ref{lem: curvatura piana} to the curve $\Gamma$ on the plane $\Pi$, deducing that $\Gamma$ cannot meet $B_{1}(y)$ in any other point than $x_1$, in contradiction with the existence of $x_2$. 

It remains to show that $y= x_\eps +  \nu^\eps(x_\eps)$ for some $x_\eps \in \Gamma_{1,\eps}^R$. Having proved the uniqueness of the projection, this follows directly from \cite[Theorem 4.8-(2)]{Fed} and the smoothness of $\Gamma_{1,\eps}^R$.
\end{proof}

Lemma \ref{lemma:dist_perturbedsupp} is crucial since it allows us to produce a family of admissible variations of the minimizer $\mf{u}$ in the following way. For $i\in \{1,2\}$, let $u_{i,\eps}\in H^1(S_{i,\eps})$ be such that
\[
\begin{cases}
\Delta u_{i,\eps}=0 & \text{ in } S_{i,\eps}\\
u_{i,\eps}=u_i & \text{ on } \partial S_{i,\eps}\backslash \Gamma^{R}_{i,\eps}=\partial S_i\backslash \Gamma_{i}^R\\
u_{i,\eps}=0 & \text{ on } \Gamma_{i,\eps}^R
\end{cases}
\]
extended by zero to  $\Omega\backslash S_{i,\eps}$. Observe that $S_{i,\eps}=\{x\in \Omega:\ u_{i,\eps}(x)>0\}$, and that for $\eps \ge 0$ small the vector $(u_{1,\eps}, u_{2,\eps},u_3,\ldots, u_k)$ belongs to the set $H_\infty$ --- defined in \eqref{def H infty}--- by Lemma \ref{lemma:dist_perturbedsupp}.

\begin{proposition}\label{prp shape derivative}
We have
\begin{align}
	&\left. \frac{d}{d \eps} \int_\Omega |\nabla u_{1,\eps}|^2 \right|_{\eps = 0^+} = - \int_{\Gamma_{1}^R} \eta(x) (\partial_{\nu_1}u_1)^2,  \label{eq:derivative_on_S1}\\	
	&\left. \frac{d}{d \eps} \int_\Omega |\nabla u_{2,\eps}|^2 \right|_{\eps = 0^+} = \int_{\Gamma_{2}^R} \eta(x+\nu_2(x)) (\partial_{\nu_2}u_2)^2.\label{eq:derivative_on_S2}
\end{align}
\end{proposition}
\begin{proof}The identity  \eqref{eq:derivative_on_S1} is a direct consequence of Lemma \ref{lemma:shapederivative_boundary} in the appendix, with $S:=S_1$ and $\omega=B_R(x_0)$, since
\[
Y_1:= \left. \frac{d}{d\eps}F_{1,\eps}(x) \right|_{\eps=0}  =\eta(x) \nu_1(x).
\]

As for \eqref{eq:derivative_on_S2}, we apply the same lemma with $S=S_2$ and $\omega=B_{y_0}^R$. We have
\[
Y_2(y):= \left.\frac{d}{d\eps} F_{2,\eps} (y) \right|_{\eps=0} = \eta(y+\nu_2(y)) \nu_1(y+\nu_2(y)) + \left. \frac{d}{d\eps}\nu^\eps(y+\nu_2(y)) \right|_{\eps=0}
\]
for every $y\in \Gamma_{2}^R$. Recalling \eqref{eq:relation_between_normals} and taking into account \eqref{eq:derivative_normal_orthogonal}, we have
\[
\left\langle \left. \frac{d}{d\eps}\nu^\eps(y+\nu_2(y)) \right|_{\eps=0}  ,\nu_2(y) \right \rangle = \left\langle \left. \frac{d}{d\eps}\nu^\eps(y+\nu_2(y)) \right|_{\eps=0}  ,-\nu_1(y+\nu_2(y))\right\rangle=0.
\]
Therefore, using \eqref{eq:relation_between_normals} once again, $\langle Y_2(y), \nu_2(y) \rangle = \eta(y+\nu_2(y))$, 
and \eqref{eq:derivative_on_S2} follows by Lemma \ref{lemma:shapederivative_boundary}. 
\end{proof}

\begin{proof}[Proof of Theorem \ref{thm:curvaturerelations}] Without loss of generality we work in the case $i=1$ and $j=2$, and use the notations previously introduced. Take, for $\eps\ge 0$ small, the vector $(u_{1,\eps}, u_{2,\eps},u_3,\ldots, u_k)$, which by Lemma \ref{lemma:dist_perturbedsupp} belongs to the set $H_\infty$. Since $u_{1,0}=u_1$ and $u_{2,0}=u_2$, then by the minimality of $\mf{u}$ we have that
\[
\left.\frac{d}{d\eps} J_\infty(u_{1,\eps},u_{2,\eps},u_3,\ldots, u_k) \right|_{\eps=0^+} \ge 0.
\] 
By Proposition \ref{prp shape derivative}, this is equivalent to
\[
\int_{\Gamma_{1}^R} \eta(x) (\partial_{\nu_1}u_1)^2 \le \int_{\Gamma_{2}^R} \eta(x+\nu_2(x)) (\partial_{\nu_2}u_2)^2.
\]
This identity holds true for every nonnegative $\eta\in \mathcal{C}^\infty_{\rm c}(B_R(x_0))$. In particular, by taking $\eta=\eta_\delta$ such that $\eta_\delta(x)=1$ for $x\in B_{R-2\delta}(x_0)$ and $\eta_\delta(x)=0$ in $B_{R}(x_0)\backslash B_{R-\delta}(x_0)$, and by making $\delta\to 0$, we can easily conclude that
\[
\int_{\Gamma_{1}^R} (\partial_{\nu_1}u_1)^2\le \int_{\Gamma_{2}^R} (\partial_{\nu_2}u_2)^2.
\]
Arguing exactly in the same way, but deforming first $\Gamma_{2,R}$, and afterwards $\Gamma_{1,R}$, we can prove that also the opposite inequality holds, and hence
\[
\int_{\Gamma_{1}^R} (\partial_{\nu_1}u_1)^2=\int_{\Gamma_{2}^R} (\partial_{\nu_2}u_2)^2.
\]
Therefore
\begin{equation}\label{eq:curvaturerelations_aux1}
\frac{\fint_{\Gamma_{1}^R} (\partial_{\nu_1}u_1)^2}{\fint_{\Gamma_{2}^R} (\partial_{\nu_2}u_2)^2}= \frac{|\Gamma_{2}^R|}{|\Gamma_{1}^R|},
\end{equation}
and we can thus end the proof by applying \cite[Lemma 9.3]{CaPaQu}, which states that the right-hand-side of \eqref{eq:curvaturerelations_aux1} tends to the right-hand-side of \eqref{free-bound cond} as $R\to 0$. We point out that, with respect to \cite{CaPaQu}, the modulus is present in our formula \eqref{free-bound cond}. This is only a consequence of the different convention that we adopted regarding the sign of the curvatures.
\end{proof}


\section{Existence and properties of solutions to problem (B)}\label{sec:ProblemB}

We focus now on problem (B). It is convenient to restate the problem as follows. Letting, for all $\bu \in H^1_0(\Omega; \R^k)$,
\[
J(\bu) = F\left(\int_\Omega|\nabla u_1|^2, \dots, \int_\Omega |\nabla u_k|^2 \right),
\]
we define
\begin{equation}\label{eqn eig dist}
c:= \inf_{ \bu \in H_{\infty}} J(\bu)
\end{equation}
where
\[
H_{\infty}=\left\{\mf{u} = (u_1,\dots,u_k) \in H^1(\Omega,\R^k)\left| \begin{array}{c} \dist(\supp u_i,\supp u_j)\ge 1\quad  \forall i\neq j \\ 
\displaystyle \int_\Omega u_i^2 = 1 \ \forall i
\end{array}\right. \right\}.
\]
Clearly, since to each set $\omega_i$ of an element in $\mathcal{P}_{k}$ we can associate an eigenvalue $u_i\in H^1_0(\omega_i)$, we have
\[
c\le 	\inf_{ (\omega_1, \dots, \omega_k) \in \mathcal{P}_{k}(\Omega) } F(\lambda_1(\omega_1), \dots, \lambda_1(\omega_k)).
\]
We show below that these levels coincide.
\subsection{Existence of a minimizer and its first properties}
We first address the problem of existence of optimal partitions, and derive some preliminary properties of the sets composing the minimal solutions. This part is close the results in Section \ref{sec existence} and for this reason we shall only give a brief sketch of the methodology.

We consider the auxiliary problem: for any $\bu \in H^1_0(\Omega, \R^k)$ we let
\begin{equation}\label{eqn eig penalized}
J_{\beta}(\bu) = F\left(\int_{\Omega} |\nabla u_1|^2, \dots, \int_{\Omega} |\nabla u_k|^2 \right) + \sum_{1 \le i < j \le k} \iint_{\Omega \times \Omega} \beta \mathds{1}_{B_1}(x-y) u_i^2(x) u_j^2(y) \, dx\, dy.
\end{equation}
We have, similarly to Theorem \ref{prop:existence}:

\begin{theorem}\label{thm eig existence}
	For every $\beta>0$, there exists a nonnegative minimizer $\mf{u}_{\beta}=(u_{1,\beta},\ldots, u_{k,\beta})$ of $J_{\beta}$ in the set
	\begin{equation}\label{eqn L2 constraint}
	H:= \left\{ \mf{u}=(u_{1},\ldots, u_{k}) \in H^1_0(\Omega, \R^k) : \int_{\Omega} u_i^2 = 1 \qquad \forall i = 1, \dots, k \right\}.
	\end{equation}
	There exist $\mu_{1,\beta}, \dots, \mu_{k,\beta}>0$ such that  $\mf{u}_{\beta}$ is a nonnegative solution of
	\begin{equation}\label{eqn system eig pen}
	- \partial_i F\left(\int_{\Omega} |\nabla u_{1}|^2, \dots, \int_{\Omega} |\nabla u_{k}|^2 \right) \Delta u_{i} = \mu_{i,\beta} u_{i} - \beta u_{i}\sum_{j\neq i} \left(\mathds{1}_{B_r} \star u_j^2 \right).
	\end{equation}
Moreover, the family $\{\bu_{\beta}: \beta >0\}$ is uniformly bounded in $H^1_0 \cap L^\infty(\Omega,\R^k)$,
and	there exists $\mf{u}=(u_{1},\ldots, u_{k}) \in H$ such that:
	\begin{enumerate}
		\item $\bu_{\beta} \to \bu$ strongly in $H^1(\Omega,\R^k)$ as $\beta\to +\infty$, up to a subsequence;
		\item $\dist(\supp u_{i}, \supp u_{j})\ge 1$, for every $i\neq j$, so that $\bu\in H_{\infty}$;
		\item for every $i\neq j$, 
		\[
		\lim_{\beta\to +\infty} \iint_{\Omega\times \Omega} \mathds{1}_{B_1}(x-y) u_{i,\beta}^2 (x)u_{j,\beta}^2(y) \, dx\,dy=0;
		\]
		\item $\bu$ is a minimizer for $c$, defined in \eqref{eqn eig dist}.\end{enumerate}
\end{theorem}
\begin{proof}
All the listed properties can be shown by very similar arguments of Theorem \ref{prop:existence}, we shall only consider here those that are new. In particular, we focus on the uniform bounds on $\{\mf{u}_{\beta}\}$.

The existence of a nonnegative minimizer $\mf{u}_{\beta}$ for $J_{\beta}$ on $H$ is given by the direct method of the calculus of variations ($J$ is lower-semicontinuous because $F$ is component-wise increasing). Since $H_{\infty}$ is not empty, it contains a smooth function $\mf{v}=(v_1,\ldots, v_k)$. Thus, $J_{\beta}(\mf{u}_{\beta}) \le c \le J(\mf{v})<+\infty$ for every $\beta >0$, and this implies that $\{\mf{u}_{\beta}, \beta>0\}$ is bounded in $H^1_0$.  Notice also that, by definition,
\[
	\int_\Omega |\nabla u_{i,\beta}|^2 \ge \lambda_1(\Omega) \qquad \text{for any $i = 1, \dots, k$ and $\beta >0$}.
	\]
Therefore, by the assumptions on $F$, there exists $a > 0$ such that
\[
a < \partial_i F\left(\int_{\Omega} |\nabla u_{1,\beta}|^2, \dots, \int_{\Omega} |\nabla u_{k,\beta}|^2 \right) < \frac{1}{a} \qquad  \text{for any $i = 1, \dots, k$ and $\beta >0$}.
\]
It follows, by the method of the Lagrange multipliers, that any minimizer $\mf{u}_{\beta}$ is a weak solution to \eqref{eqn system eig pen}. Testing such equations by $\mf{u}_{\beta}$ itself and using the uniform bound on $J_\beta(\mf{u}_{\beta})$, we obtain that the exists $\mu > 0$ such that
	\[
	0 < \mu_{i,\beta} < \mu \quad  \text{for any $i = 1, \dots, k$ and $\beta >0$}. 
	\]
	The proof of the uniform $L^\infty$ bounds is then a rather standard consequence of the Brezis-Kato iteration technique, since $-\Delta u_{i,\beta} \leq \mu u_{i,\beta}$. 
The remaining properties can be shown reasoning exactly as in the proof of Theorem \ref{prop:existence}.
\end{proof}

The previous result shows the existence of minimizers for problem $c$, in connection with an elliptic system with long-range competition. Since both $H_{\infty}$ and $J$ are invariant under the transformation $(u_1,\ldots, u_k)\mapsto (|u_1|,\ldots, |u_k|)$, we can work from now on, without loss of generality, with nonnegative functions. In what follows, we will show that all the minimizers for $c$ are continuous (actually, we will show that they are Lipschitz continuous in $\Omega$), and this will imply that \eqref{eqn part dist} and \eqref{eqn eig dist} coincide, and there is a one-to-one correspondence between (open) optimal partitions $(\omega_1,\ldots, \omega_k)$ of  \eqref{eqn part dist} and minimizers $\bu$ of \eqref{eqn eig dist}: for every $\bu$ minimizer of $c$, the sets $\omega_{i}=\{u_{i}>0\}$ constitute an optimal partition at distance 1 of $\Omega$.


\subsection{Proof of Theorems \ref{thm: properties minimizers} and \ref{thm: improved reg} for problem (B)}

By following exactly the same lines of the proof of Theorem \ref{thm: properties minimizers}, (1)--(2)--(3), (5)--(6) for problem (A), we can show the exact same properties for any minimizer $\mf{u}$ of the level $c$.  

Regarding the regularity of the eigenfunctions, using the notations of Section \ref{sec:Properties_of_minimizers}, we observe that $\mf{u}=0$ on $\pa \Omega$, and that $\Omega$  satisfies the $r$-uniform exterior sphere condition for some $r>0$. Then the Lipschitz continuity in $\overline{\Omega}$ is a direct application of Theorem \ref{thm: lip general} with $f=\lambda_1(\omega_i) u_i$, $\Lambda=\R^N$ and $A:=A_i$ (this shows Theorem {\ref{thm: improved reg}).

Observe that the continuity of $\mf{u}$ implies that then $\omega_{i}=\{u_{i}>0\}$, $i=1,\ldots, k$ are minimizers for problem (B). Thus $c$ and \eqref{eqn part dist} coincide, and given any optimal partition of \eqref{eqn part dist}, then the conclusions of Theorem \ref{thm: properties minimizers} hold also for the associated eigenvalues $\bu$. 

\subsection{Proof of Theorem \ref{thm:curvaturerelations} for problem (B)}

The proof of this result for problem (B) follows word by word the lines of the proof for problem (A), replacing only Lemma \ref{lemma:shapederivative_boundary} by the classical Hadamard's variational formula \cite[Theorem 2.5.1]{Hen}.

\appendix

\section{Shape Derivatives}\label{appendix}

In this appendix we establish a formula which relates the change of the energy of the harmonic extension of a function $\varphi$, defined on a boundary $\pa S$ and vanishing on a portion $\pa S \cap \omega$ of $\pa S$. The domain variation is localized on $\pa S \cap \omega$. Although similar results are by now well known, and excellent references are available (we refer for instance to \cite[Chapter 5]{HePi}), we could not find exactly the result we needed, and therefore we provide here a short discussion for the sake of completeness.
 
 \medskip
 
Let $S \subset \R^N$ be a open set, and let $\omega\subset \R^N$ be a bounded smooth domain  such that $\partial S\cap \textrm{int}\,(\omega) \neq \emptyset$. For a function $\varphi:\partial S \to \R $ such that  $\varphi \in Lip(\partial S)$ and $\varphi(x) = 0$ if $x \in   \partial S\cap \overline \omega$, we consider its harmonic extension in $S$, that is the function $u \in H^1(S)$ solution to
\[
	\begin{cases}
		\Delta u = 0 &\text{in $S$}\\
		u = \varphi &\text{on $\partial S$}
	\end{cases}\quad  \text{ or, equivalently, }\quad 	\int_{S} |\nabla u| ^2 = \min\left\{ \int_{ S} |\nabla v| ^2 :
	\begin{array}{l}
		v \in H^1( S), \\
		v = \varphi \text{ on $\partial S$}
			\end{array}\right\}.
\]
The question we want to address is how a smooth deformation of a regular part of $\partial S$ where $u=0$ impacts the energy of the corresponding harmonic extension.  We start by analyzing the derivative with respect to a global homotopy $F : [0,T) \times \R^N \to\R^N$, for some $T>0$, 
satisfying:
\begin{enumerate}
\item[(H1)] $t\in [0,T) \mapsto F(t,\cdot)\in W^{1,\infty}(\R^N,\R^N) $ is differentiable at 0;
\item[(H2)] $F(0,\cdot)=\id$;
\item[(H3)] $F(t,x)=x$ for every $t\in [0,T)$, $x\in \partial S\backslash \omega$.
\end{enumerate} 
For notation convenience, we let $F_t(x) = F(t,x)$, 
while $DF_t(x):=D_xF(t,x)$. We can assume that $T>0$ is sufficiently small so that $D_xF(t,x)$ is an invertible matrix for $(t,x) \in [0,T[\times \R^N$. Moreover, we define
\[
Y= F'_0:=\left. \frac{d}{d t}F_t(\cdot) \right|_{t=0}\in W^{1,\infty}(\R^N,\R^N),
\]
so that, by (H1), $F_t(x)= x+ t Y(x) + \textrm{o}(t)$  in  $W^{1,\infty}(\R^N,\R^N)$, as $t\to 0$.

For every $t \in [0,T)$ we let $S_t = F_t(S)$ and $\Gamma_t=F_t(\partial S\cap \omega)$. Let $u_t \in H^1(S_t)$ be such that
\[
	\begin{cases}
		\Delta u_t = 0 &\text{in $S_t$}\\
		u = \varphi &\text{on $\partial S\backslash \omega$}\\
		u = 0 & \text{on $\Gamma_t$}
	\end{cases} \ \ \text{that is}\ \ 	I_t := \int_{S_t} |\nabla u_t| ^2 = \min\left\{ \int_{S_t} |\nabla v| ^2 :
	\begin{array}{l}
		v \in H^1(S_t), \\
		v = \varphi \text{ on $\partial S \backslash \omega$}, \\
		v = 0 \text{ on $\Gamma_t$}
	\end{array}\right\}
\]
\begin{lemma}\label{lem first shape}
Under the previous assumptions, the function $I_t$ is differentiable at $t=0$, with
\[
	\left. \frac{d}{d t} I_t \right|_{t = 0} = \int_{S} \langle (\div Y\, \id - 2 D Y) \nabla u, \nabla u\rangle
\]
\end{lemma}

\begin{proof} \textbf{Step 1:}  Fixing the domain through a change of variables. For any $t \in [0,T[$, let $v_t \in H^1(S)$ be defined as $v_t := u_t \circ F_t$. Observe that for every $v\in H^1( S_t)$ one has
\[
\int_{S_t}|\nabla v(y)|^2\, dy=\int_{F_t( S)}|\nabla v(y)|^2\, dy=\int_{S} |[(DF_t(x))^{-1}]^T\nabla (v(F_t(x))|^2 \det(DF(x)) \, dx.
\]
Thus $v_t$ is the minimizer of
\[
	I_t = \min\left\{ \int_{S} \det(D F_t ) |[(D F_t)^{-1}]^T \nabla w|^2  :
	\begin{array}{l}
		w \in H^1(S), \\
		w = \varphi \text{ on $\partial S$}
	\end{array}\right\}
\]
(recall that $\varphi=0$ on $\partial S\cap \omega$)
and a solution to the problem
\[
	\begin{cases}
		-\div(A_t \nabla v_t) = 0 &\text{in $S $}\\
		v_t = \varphi &\text{on $\partial S$}
	\end{cases}
\]
with $A_t(x) = \det(D F_t(x)) (D F_t(x))^{-1}[(D F_t(x))^{-1}]^T$. Observe that $A_t(x)$ is symmetric and there exist $0 < \lambda < \Lambda$ such that
\[
\lambda |\xi|^2 \le \langle A_t(x) \xi, \xi \rangle \le \Lambda |\xi|^2 \qquad \text{for all $x \in \R^N, t \in [0,T), \xi \in \R^N$};
\]
the map	$t\in [0,T)\mapsto A_t\in L^\infty(\R^N)$ is differentiable at $t=0$,   and
$		\lim_{t \to 0} A_t =A_0 =\id$ uniformly in $\R^N$;
		and by  recalling that $Y := F'_0$, we have by Jacobi's formula
\[
\left. \frac{d}{d t}A_t(x)\right|_{t = 0}=\div Y \,\id - (D Y + D Y^T) \qquad \text{uniformly in $\R^N$. }
	\]

\noindent \textbf{Step 2:}  Differentiability of the map $t\in [0,T) \mapsto v_t\in H^1(S)$ at $t=0$. We introduce the incremental quotients
\[
	w_{t,0} := \frac{v_t - v_0}{t-0}=\frac{v_t-u}{t} \in H_0^1( S), \qquad t\in ]0,T[.
\]
Each $w_{t,0}$ is a solution to
\begin{equation}\label{eqn wt}
	\begin{cases}
		-\div(A_t \nabla w_{t,0}) = \div\left( \frac{A_t - \id}{t} \nabla u\right) &\text{in $S$}\\
		w_{t,0} = 0 &\text{on $\partial S$}.
	\end{cases}
\end{equation}
We introduce the function $w_0 \in H^1_0(S)$ solution to
\begin{equation}\label{eqn w0}
	\begin{cases}
		-\Delta w_0 = \div(A_0' \nabla u) &\text{in $ S$}\\
		w_0 = 0 &\text{on $\partial S$}.
	\end{cases}
\end{equation}
and show that indeed $w_{t,0}\to w_0$ as $t\to 0$, strongly in $H^1_0(S)$, so that $t \mapsto v_t$ is differentiable at $t=0$, with $v_0'=w_0$. To do this, we subtract \eqref{eqn w0} from \eqref{eqn wt} and obtain the identity
\[
	-\div(A_t \nabla(w_{t,0} -w_0)) = \div((A_t - A_0) \nabla w_0) + \div\left( \left(\frac{A_t - A_0}{t} - A_0'\right) \nabla v_0\right)
\]
Testing this equation by $w_{t,0} - w_0\in H^1_0(S)$, we can conclude that
\[
	\left(\int_{S} |\nabla (w_{t,0} - w_0)|^2\right)^\frac12\le \frac{1}{\lambda} \left( \|A_t-A_0\|_{\infty} \|w_0\|_{H^1} +  \left\|\frac{A_t - A_0}{t} - A_0'\right\|_{\infty} \|v_0\|_{H^1} \right)
\]
and the claim follows recalling the properties of the functions $A_t$. 

\smallbreak

\noindent \textbf{Step 3:}  Differentiability of the map $t\in [0,T)\mapsto I_t\in \R$ at $t=0$. As a result of the previous step, the derivative of $I_t$ at $t=0$ is equal to
\[
 \lim_{t \to 0}  \int_{S} \left( \left\langle \frac{A_t - A_0}{t} \nabla v_t+ A_0 \nabla w_{t,0}, \nabla v_t \right\rangle + \left\langle A_0 \nabla v_0, \nabla w_{t,0} \right\rangle \right)= \int_{S} \langle A_0' \nabla u, \nabla u\rangle + 2 \int_{S}\langle \nabla w_0, \nabla u\rangle\\
\]
By testing the equation of $u$ by $w_0\in H^1_0(S)$, we see that the last term in the previous expression is zero, and by exploiting the symmetry of the scalar product we obtain
\[
	\left. \frac{d}{d t} I_t \right|_{t = 0} = \int_{S} \langle A_0' \nabla u, \nabla u\rangle   = \int_{S} \langle \left(\div Y \id - 2 D Y \right) \nabla u, \nabla u\rangle \qedhere
\]
\end{proof}

We now show that, if $F_t$ leaves invariant a neighborhood of $\partial S \backslash \omega$, then the derivatives in Lemma \ref{lem first shape} can be expressed only in terms of the value of the first order behavior of $F$ around $\partial S \cap \omega$.
\begin{lemma}\label{lemma:shapederivative_boundary}
Assume $(H1)$,$(H2)$, and instead of $(H3)$ assume the stronger condition 
\begin{itemize}
\item[(H3')] $F(t,x)=x$ for every $t\in [0,T)$, $x\in S\backslash \omega' $, for some $\omega'\Subset  \omega$;
\end{itemize} 
and assume also that $\partial S\cap \omega$  is a smooth hypersurface Then we have
\[
	\left. \frac{d}{d t} I_t \right|_{t = 0} =- \int_{\omega \cap \partial S}  (Y \cdot \nu) (\partial_\nu u)^2
\]
In particular, the first derivative of the energy at 0, $I'_0$, depends on $F_t$ only through the value of $Y = F'_0 $ over $\omega\cap \partial S$.
\end{lemma}
\begin{proof}
Observe that the assumptions imply that $Y\in W^{1,\infty}(\R^N)$ satisfies $Y=0$ in $S\setminus  \omega$. Moreover, since $u$ is harmonic in $S$, $u\in H^2(O)$, for every $O\Subset \omega\cap S$. Thus we can test the equation of $u$ with  $Y \cdot \nabla u\in H^1(S)$, obtaining
\begin{multline*}
	0 = \int_{ S} \nabla u \cdot \nabla (Y \cdot \nabla u) - \int_{\omega \cap \partial  S} (Y \cdot \nabla u)( \nu \cdot \nabla u )\\
	= \int_{\omega \cap S} \left(\langle \nabla u, D Y \nabla u \rangle +  \langle \nabla u,  D^2 u  Y \rangle \right)- \int_{ \omega \cap \partial S} (Y \cdot \nabla u)( \nu \cdot \nabla u )\\
	=  \int_{\omega \cap S} \left( \langle \nabla u, D Y \nabla u \rangle +  \frac12 \langle \nabla |\nabla u|^2,Y\rangle \right) - \int_{\omega \cap \partial S} (Y \cdot \nabla u)( \nu \cdot \nabla u )
\end{multline*}
(the boundary term is well defined since $\omega \cap \partial S$ is a smooth hypersurface). A further integration by parts and the observation that, since $u=0$ on $\omega \cap \partial S$, we have $|\nabla u| = |\partial_\nu u|$ and $\nabla u= (\nu \cdot \nabla u )\nu$ on $\omega \cap \partial S$, yields the identities 
\[
	\int_{\omega \cap S}\langle (\div Y \,\id - 2 D Y) \nabla u, \nabla u\rangle = \int_{\omega \cap \partial S}  \left( (Y \cdot \nu) |\nabla u|^2- 2(Y \cdot \nabla u)( \nu \cdot \nabla u )\right)=-\int_{\omega \cap \partial S}  (Y \cdot \nu) (\partial_\nu  u)^2.\qedhere
\]
\end{proof}

\bigbreak

\textbf{Acknowledgments.} The authors are partially supported by the ERC Advanced Grant 2013 n. 339958 ``Complex Patterns for Strongly Interacting Dynamical Systems - COMPAT''.

N. Soave is partially supported by the PRIN-2015KB9WPT\texttt{\char`_}010 Grant: ``Variational methods, with applications to problems in mathematical physics and geometry".

H. Tavares is partially supported by FCT - Portugal through the project PEst-OE/EEI/LA0009/2013.

A. Zilio is partially supported by the ERC Advanced Grant 2013 n. 321186 ``ReaDi -- Reaction-Diffusion Equations, Propagation and Modelling''.

\end{document}